\documentclass[12pt]{amsart}

\usepackage{fullpage}
\usepackage{amssymb, hyperref}
\usepackage{amsthm}
\usepackage{graphicx}
\usepackage[all,2cell]{xy}
\usepackage{verbatim}
\usepackage{tikz}
\usepackage{tikz-cd}
\usepackage{hyperref}
\usetikzlibrary{matrix,arrows,decorations.pathmorphing}
\numberwithin{equation}{section}
\newtheorem{theorem}{Theorem}[section]
\newtheorem{corollary}[theorem]{Corollary}

\newtheorem{lemma}[theorem]{Lemma}

\newtheorem{proposition}[theorem]{Proposition}

\begin{document}

\author{B. Narasimha Chary, S. Senthamarai Kannan and A.J. Parameswaran}

\address{Chennai Mathematical Institute, Plot H1, SIPCOT IT Park, 
Siruseri, Kelambakkam,  603103, India.}

\email{chary@cmi.ac.in.}

\address{Chennai Mathematical Institute, Plot H1, SIPCOT IT Park, 
Siruseri, Kelambakkam, 603103, India.}

\email{kannan@cmi.ac.in.}

 \address{Tata Institute of Fundamental Research, Homi Bhabha Road, Colaba, Mumbai,
 500004, India.} 

 \email{param@math.tifr.res.in}

\title{Automorphism group of a Bott-Samelson-Demazure-Hansen Variety}

\begin{abstract}
Let $G$ be a simple, adjoint, algebraic group over the field of complex numbers, 
$B$ be a Borel subgroup of $G$ containing a maximal torus $T$ of $G$, $w$ be an element 
of the Weyl group $W$ and $X(w)$ be the Schubert variety in $G/B$
corresponding to $w$. Let $Z(w,\underline i)$ be the Bott-Samelson-Demazure-Hansen 
variety corresponding to a reduced expression $\underline i$ of $w$. 

In this article, we compute the connected component 
$Aut^0(Z(w, \underline i))$ 
of the automorphism group of  $Z(w,\underline i)$ containing the identity automorphism. We show 
that
$Aut^0(Z(w, \underline i))$ contains a closed subgroup isomorphic to  $B$ if and only if
$w^{-1}(\alpha_0)<0$, where $\alpha_0$ is the highest root.
If $w_0$ denotes the longest element of $W$, then we prove 
that $Aut^0(Z(w_0, \underline i))$ is a parabolic subgroup of $G$. It is also 
shown that this parabolic subgroup depends very much on the chosen reduced 
expression $\underline i$ of $w_0$ and we describe all parabolic subgroups of $G$ that occur as 
$Aut^0(Z(w_0, \underline i))$. 
If $G$ is simply laced, then we show that for every $w\in W$, and for every reduced expression $\underline i$
 of $w$, $ Aut^0(Z(w, \underline i))$
is a quotient of the 
parabolic subgroup $Aut^0(Z(w_0, \underline j))$ of $G$ for a suitable choice of a reduced 
expression $\underline j$ of $w_0$ (see Theorem \ref{cor3}).

\end{abstract}

\maketitle


\section{Introduction}

Let $G$ be a simple algebraic group over  the field $\mathbb{C}$ of 
complex numbers of adjoint type. We fix a maximal torus $T$ of $G$ and 
let $W = N_G(T)/T$ denote the Weyl group of $G$ with respect to $T$. We denote 
  the set of roots of $G$ with respect to $T$ by $R$. Let $B^{+}$ be a  Borel subgroup of $G$ 
  containing $T$. Let $B$ be the Borel subgroup of $G$ opposite to $B^{+}$ determined by $T$. 
  That is, $B=n_{0}B^{+}n_{0}^{-1}$, where $n_{0}$ is a representative in $N_{G}(T)$ of the longest     element $w_{0}$ of $W$. Let  $R^{+}\subset R$ be 
  the set of positive roots of $G$ with respect to the Borel subgroup $B^{+}$. Note that the set of 
  roots of $B$ is equal to the set $R^{-} :=-R^{+}$ of negative roots.
We use the notation $\beta>0$ for $\beta \in R^+$ and $\beta<0$ for 
$\beta \in R^{-}$.
Let $S = \{\alpha_1,\ldots,\alpha_n\}$ 
denote the set of all simple roots in $R^{+}$, where $n$ is the rank of $G$. 
The simple reflection in the Weyl group corresponding to a simple root $\alpha$ is denoted
by $s_{\alpha}$. For simplicity of notation, the simple reflection corresponding to a simple root $\alpha_{i}$ is denoted
by $s_{i}$.  For any simple root $\alpha$, we denote the fundamental weight
corresponding to $\alpha$ by $\omega_{\alpha}$. 
Let $\alpha_{0}$ denote the highest root and $\rho$ denote the half sum of all 
positive roots, which is also same as the sum of all fundamental weights.

For $w \in W$, let $X(w):=\overline{BwB/B}$ denote the Schubert variety in
$G/B$ corresponding to $w$. Given a reduced expression 
$w=s_{{i_1}}s_{{i_2}}\cdots s_{{i_r}}$ of $w$, with 
the corresponding tuple $\underline i:=(i_1,\ldots,i_r)$, we denote by 
$Z(w,{\underline i})$ the desingularization  of the Schubert variety $X(w)$, 
which is now known as the Bott-Samelson-Demazure-Hansen variety. It was 
first introduced by Bott and Samelson in a differential geometric and 
topological context (see \cite{Bott-Sam}). Demazure in  \cite{Dem1} and 
Hansen in \cite{Hansen} independently adapted the construction in 
algebro-geometric situation, which explains the reason for the name. 
For the sake of simplicity, we will denote any Bott-Samelson-Demazure-Hansen 
variety by a BSDH-variety. 

The construction of the BSDH-variety $Z(w,{\underline i})$ depends on the 
choice of the reduced expression $\underline i$ of $w$. So, it is natural to ask that for a 
given $w\in W$ whether the BSDH-varieties corresponding to two different 
reduced expressions of $w$ are isomorphic? This article deals  with the study
of the automorphism group of the BSDH-varieties in order to answer this question.

We recall the  following notation before describing the main results:

Let $\mathfrak{g}$ denote the Lie algebra of $G$, $\mathfrak{h}$ be the 
Lie algebra of $T$, $\mathfrak{b}$ be the Lie algebra of $B$. Let $X(T)$ 
denote the group  of all characters of $T$.

We have $X(T)\otimes \mathbb{R}=Hom_{\mathbb{R}}(\mathfrak{h}_{\mathbb{R}}, \mathbb{R})$, 
the dual of the real form $\mathfrak h_{\mathbb{R}}$ of $\mathfrak{h}$.
The positive definite $W$-invariant bilinear form on 
$Hom_{\mathbb{R}}(\mathfrak{h}_{\mathbb{R}}, \mathbb{R})$ induced by the Killing form of $\mathfrak{g}$
is denoted by $(~,~)$. 
We use the notation $\left< ~,~ \right>$ to
denote $\langle \nu, \alpha \rangle  = \frac{2(\nu,
\alpha)}{(\alpha,\alpha)}$ for $\nu \in X(T)\otimes \mathbb{R}$ and $\alpha \in R$.

Given a reduced expression $w=s_{{i_1}}s_{{i_2}}\cdots s_{{i_r}}$, 
let $\underline i:=(i_1,\ldots,i_r)$.
Set $$J^{'}(w, \underline i):=\{l\in \{1,2, \ldots, r\}:\langle \alpha_{i_l}, \alpha_{i_k}
\rangle = 0~{\rm 
for ~all }~ k < l\}$$ $$J(w, \underline i):=\{\alpha_{i_l}: l\in J^{'}(w, \underline i)\}
\subset S.$$ Note that the simple reflections $\{s_{i_j}: j\in J'(w, \underline i)\}$ commute with each other.
Let $ W_{J(w, \underline i)}$ be the subgroup of $W$ generated by 
$\{s_j\in W\mid \alpha_j\in J(w, \underline i)\}$. Let 
$$P_{J(w, \underline i)}:=BW_{J(w, \underline i)}B$$
be the corresponding standard parabolic subgroup of $G$. By abuse of notation, here 
$W_{J(w, \underline i)}$ in the definition of the parabolic subgroup $P_{J(w, \underline i)}$ 
means any lift of elements of $W_{J(w, \underline i)}$ to $N_G(T)$. Let $N=|R^{+}|$.
Further, let $Aut^0(Z(w, \underline i))$ be the connected component 
of the automorphism group of $Z(w, \underline i)$ containing the identity automorphism.

The main results of this article are (see Theorem \ref{cor3}):

 \begin{enumerate}
 
 \item For any reduced expression $\underline i$ of $w_0$, $ Aut^0(Z(w_0, \underline i)) 
 \simeq P_{J(w_0, \underline i)} $.
  \item For any reduced expression $\underline i$ of $w$, $Aut^0(Z(w, \underline i))$ 
 contains a closed subgroup isomorphic to $ P_{J(w, \underline i)}$  if and only if 
 $w^{-1}(\alpha_0)<0$. In such a case, $P_{J(w, \underline i)}=P_{J(w_0, \underline j)}$ for any reduced expression $w_0=s_{{j_1}}s_{{j_2}}\cdots s_{{j_N}}$ of $w_0$ such that 
$\underline j=(j_1, j_2, \ldots, j_N)$ and $(j_1, j_2, \ldots, j_r)=\underline i$.
 \item If $G$ is simply laced,    
$Aut^0(Z(w, \underline i))$ is a quotient of $Aut^0(Z(w_0, \underline j))$, where $\underline j$ is as in (2).
\item If $G$ is simply laced, $Aut^0(Z(w, \underline i))\simeq P_{J(w, \underline i)}$  if and only if
 $w^{-1}(\alpha_0)<0$. In such a case, we have $P_{J(w, \underline i)}=P_{J(w_0, \underline j)}$ where $\underline j$ is as in (2).
\item The rank of $Aut^0(Z(w, \underline i))$ is at most the rank of $G$.
 \end{enumerate}

Consider the left action of $T$ on $G/B$. Note that $X(w)$ is $T$-stable. Since $T$ is a 
reductive group,
studying the semi-stable points of $X(w)$ for $T$-linearized line bundles is an interesting 
problem related to Geometric Invariant Theory.
By \cite[Lemma 2.1]{KP}, the condition $w^{-1}(\alpha_0)<0$ is equivalent to the Schubert
variety $X(w^{-1})$
having semi-stable points for the choice of the $T$-linearized line bundle 
$\mathcal{L}_{\alpha_0}$ associated to $\alpha_0$.
Corollary \ref{semi} is a formulation of the main results using semi-stable points.

The paper is organized as follows:

In Section $2$, we recall the definition of the
BSDH-varieties and some results on the cohomology of line bundles on Schubert
varieties. The main results used here are the results of Demazure (\cite{Dem1} 
and \cite{Dem}). A structure theorem for indecomposable $B_\alpha$-modules  is 
recalled from \cite{Ka1}, where $B_{\alpha}$ is the intersection of $B$ and the Levi subgroup
of the minimal parabolic subgroup of $G$ containing $B$ corresponding to $\alpha\in S$.
The important results recalled are from 
\cite{Ka4}, which states that all $i^{th}$ cohomology groups of ${\mathcal L}_\beta$ 
vanish on $X(w)$ for all $i\geq 2$ and for all $w\in W$ and for any positive root
$\beta$. In the simply laced case in fact these cohomology groups vanish for all $i>0$.

Section $3$ begins with a detailed description of the BSDH-varieties as iterated
${\mathbb P}^1$-bundles. Using the results of \cite{Ka4}, we conclude that higher 
cohomology groups (that is $i>1$ in general and $i>0$ in the simply laced case) of the tangent 
bundle of the BSDH-variety vanish (see Proposition \ref{prop3.1}). This implies that the BSDH-varieties are rigid 
for simply laced groups and their deformations are  unobstructed in general.

 Next three sections are more technical sections.
 
 Section $4$ is devoted to detailed computations involving the structure of $H^0$ and 
$H^1$ of the relative tangent bundle on $Z(w,\underline i)$, where $w=s_{i_1}\cdots s_{i_r}$ is a reduced
expression for $w$ and $\underline i=(i_1,\ldots, i_r)$.
 We analyze the zero weight spaces and the 
 weight spaces corresponding to  positive roots of the global sections of the relative tangent bundle and we prove 
 that these spaces
 are multiplicity free. 
 We also prove that $\mathfrak b\cap sl_{2, \alpha_{i_r}}$ is a $B_{\alpha_{i_r}}$-submodule of the global
 sections of the relative tangent bundle if and only if $X(s_{i_r})\nsubseteq X(s_{i_1}\cdots s_{i_{r-1}})$.
 While proving this, we observe that its zero weight space is at most one-dimensional (see Lemma \ref{1}). Further, we prove that $sl_{2, \alpha_{i_r}}$
is a $B_{\alpha_{i_r}}$-submodule if and only if $\langle \alpha_{i_r}, \alpha_{i_k} \rangle=0$ for all $1\leq k \leq r-1$
(see Corollary \ref{cor1}).
We conclude the section with a result on the $H^1$ of the relative tangent
 bundle.
 
 In section $5$,
 we discuss the action of the minimal parabolic subgroup $P_{\alpha_{i_1}}$ on  the
BSDH-variety $Z(w, \underline i)$.
We show that the homomorphism $f_{w_0}:\mathfrak p_{\alpha_{i_1}} \longrightarrow H^0(Z(w_0, \underline i), T_{(w_0, \underline i)})$
of Lie algebras induced by the action of $P_{\alpha_{i_1}}$ is injective
(see Lemma \ref{l3}). We also prove that $H^0(Z(w_0, \underline i), T_{(w_0, \underline i)})$ is a Lie subalgebra of $\mathfrak g$
and any Borel (respectively, maximal toral) subalgebra of
$H^0(Z(w_0, \underline i), T_{(w_0, \underline i)})$
is isomorphic to a Borel (respectively,  maximal toral) subalgebra of $\mathfrak g$
(see Corollary \ref{C1}).

In Section $6$, we study the $B$-module  of the global sections of the tangent bundle on the
BSDH-variety $Z(w, \underline i)$.
We prove that the image $f_w(\mathfrak h)$  is a maximal toral subalgebra of $H^0(Z(w, \underline i), T_{(w, \underline i)})$ (see Proposition \ref{Prop0}).
Further, we show that $sl_{\alpha_{i_j}}$ is a $B_{\alpha_{i_j}}$-submodule of $H^0(Z(w, \underline i), T_{(w, \underline i)})$
if and only if  $\langle \alpha_{i_j}, \alpha_{i_k} \rangle=0$ for all $1\leq k \leq j-1$ (see Proposition \ref{prop2}).
We conclude Section $6$ by proving that $H^0(Z(w, \underline i), T_{(w, \underline i)})$ contains a Lie subalgebra $\mathfrak b'$  isomorphic to
$\mathfrak b$ if and only if $w^{-1}(\alpha_{0})<0$ (see Proposition \ref{Borel1}). 

In Section $7$, we prove  the main results on the connected component $Aut^0(Z(w, \underline i))$ 
of the 
automorphism group of the BSDH-variety $Z(w, \underline i)$
 using the fact that the global sections $H^0(Z(w, \underline i), T_{(w, \underline i)})$ of the tangent bundle on $Z(w, \underline i)$ 
 is the Lie 
algebra of $Aut^0(Z(w, \underline i))$ . 
More precisely, we prove that the Lie algebra $\mathfrak p_{J(w_0, \underline i)}$ of $P_{J(w_0, \underline i)}$
is isomorphic to $H^0(Z(w_0, \underline i), T_{(w_0, \underline i)})$.
We also prove that for any 
reduced expression $\underline j=(j_1, j_2, \ldots, j_N)$ of $w_0$ such that $(j_1, j_2, \ldots, j_r)=\underline i$ the homomorphism $H^0(Z(w_0, \underline j), T_{(w_0, \underline j)}) \longrightarrow
H^0(Z(w, \underline i), T_{(w, \underline i)})$ of Lie algebras induced by the fibration 
$Z(w_0, \underline j)\longrightarrow Z(w, \underline i)$ is injective if and only if 
$w^{-1}(\alpha_0)<0$.
Further, we prove that if $G$ is simply laced, the homomorphism $H^0(Z(w_0, \underline j), T_{(w_0, \underline j)}) \longrightarrow
H^0(Z(w, \underline i), T_{(w, \underline i)})$ as above is surjective (see Theorem \ref{theorem1}).
We also compute the kernel of this homomorphism (see Corollary \ref{kernel}).
Using Theorem \ref{theorem1}, we prove the main results of this article.  Using Corollary \ref{kernel},
we describe the kernel of the natural homomorphism $Aut^0(Z(w_0, \underline j)) \longrightarrow Aut^0(Z(w, \underline i))$ 
of algebraic groups (see Corollary \ref{kernel1}).
Thus, we have a complete description of $Aut^0(Z(w, \underline i))$ for any reduced expression $\underline i$ of $w$
in the simply laced case.

\section{Preliminaries}\label{prelim}
Let $\{x_{\beta}: \beta \in R\}\cup \{h_{\alpha}: \alpha \in S \} $ be the 
Chevalley basis  for $\mathfrak{g}$ corresponding to the root system $R$.   
 For a simple root $\alpha$, we denote by 
$\mathfrak{g}_{\alpha}$(respectively, $\mathfrak{g}_{-\alpha}$) the one-dimensional root subspace of $\mathfrak{g}$ spanned by $x_{\alpha}$ 
(respectively, $x_{-\alpha}$). Let $sl_{2,\alpha}$ denote the $3$-dimensional 
Lie subalgebra of $\mathfrak{g}$ generated by $x_{\alpha}$ and $x_{-\alpha}$.

Let $\leq$ denote the partial order on $X(T)$ given by $\mu\leq \lambda$
if $\lambda-\mu$ is a non-negative integral linear combination of simple 
roots. We say that $\mu < \lambda$ if in addition $\lambda-\mu$ is non zero.
We set $R^{+}(w):=\{\beta\in R^{+}:w(\beta)\in R^{-}\}$.
 We refer to \cite{Hum1} and 
\cite{Hum2} for preliminaries on Lie algebras and algebraic groups.

For a simple root $\alpha\in S$, we denote by
$P_{\alpha}$ the minimal parabolic subgroup of $G$ generated by $B$ and $n_{\alpha}$, a 
lift of $s_{\alpha}$ in $N_G(T)$. 

We recall that the BSDH-variety corresponding to a reduced expression $\underline i$ of
$w=s_{{i_1}}s_{{i_2}}\cdots s_{{i_r}}$  is defined by
\[
Z(w, \underline i) = \frac {P_{\alpha_{i_{1}}} \times P_{\alpha_{i_{2}}} \times \cdots \times 
P_{\alpha_{i_{r}}}}{B \times \cdots
\times B},
\]

where the action of $B \times \cdots \times B$ on $P_{\alpha_{i_{1}}} \times P_{\alpha_{i_{2}}}
\times \cdots \times P_{\alpha_{i_{r}}}$ is given by $(p_1, \ldots , p_r)(b_1, \ldots
, b_r) = (p_1 \cdot b_1, b_1^{-1} \cdot p_2 \cdot b_2, \ldots
,b^{-1}_{r-1} \cdot p_r \cdot b_r)$, $ p_j \in P_{\alpha_{i_{j}}}$, $b_j \in B$ and 
$\underline i=(i_1, i_2, \ldots, i_r)$
(see \cite[p.73, Definition 1]{Dem1}, \cite[p.64, Definition 2.2.1]{Bri} and \cite{Hansen}).

We note that for each reduced expression $\underline i$ of $w$, $Z(w, \underline i)$ is a 
smooth projective 
variety. We denote both the natural birational surjective morphism
from $ Z(w, \underline i)$ to $X(w)$ and the composition map $Z(w, \underline i)
\longrightarrow X(w)\hookrightarrow G/B$ by $\phi_w$.

Let $f_r : Z(w, \underline i) \longrightarrow Z(ws_{i_r},
\underline i')$ denote the map induced by the
projection $$P_{\alpha_{i_1}} \times P_{\alpha_{i_2}} \times \cdots \times 
P_{\alpha_{i_r}} \longrightarrow P_{\alpha_{i_1}} \times P_{\alpha_{i_2}}
\times \cdots \times P_{\alpha_{i_{r-1}}},$$ where $i'=(i_1,i_2,\ldots, i_{r-1})$.
 We note that $f_r$ is a 
$P_{\alpha_{i_r}}/B \simeq {\mathbb P}^{1}$-fibration.

Now, we recall some preliminaries on the BSDH-varieties and some application of Leray 
spectral sequences to compute the cohomology of line bundles on Schubert varieties. 
Good references for this are \cite{Bri} and  \cite{Jan}.

  Let $L_{\alpha}$ denote the Levi subgroup of
$P_{\alpha}$ containing $T$ for $\alpha \in S$. We denote by $B_{\alpha}$ the
intersection of $L_{\alpha}$ and $B$. Then $L_{\alpha}$ is the product
of $T$ and a homomorphic image $G_{\alpha}$ of $SL(2, \mathbb C)$ via a homomorphism
$\psi: SL(2, \mathbb C)\longrightarrow L_{\alpha}$ 
(see  \cite[II, 1.3]{Jan}). 

Let $B^{\prime}_{\alpha}:=B_{\alpha}\cap G_{\alpha}\subset L_{\alpha}$. We note that 
the morphism $G_{\alpha}/B_{\alpha}^{\prime} \longrightarrow L_{\alpha}/B_{\alpha}$ 
induced by the inclusion is an isomorphism. Since $L_{\alpha}/B_{\alpha}\hookrightarrow 
P_{\alpha}/B$  is an isomorphism, to compute the cohomology groups $H^{i}(P_{\alpha}/B, V)$
for any $B$-module $V$, we treat $V$ as a $B_{\alpha}$-module  and we compute 
$H^{i}(L_{\alpha}/B_{\alpha}, V)$.

For a $B$-module $V$, let ${\mathcal L}(w,V)$ denote the restriction of the
associated  homogeneous  vector bundle on $G/B$ to $X(w)$.
By abuse of notation we denote the pull back of ${\mathcal L}(w,V)$ via $\phi_w$ to 
$Z(w, \underline i)$ also by ${\mathcal L}(w,V)$, when there is no cause 
for confusion. Then, for $j\geq 0$, we have the following isomorphism of
$B$-linearized sheaves (see \cite[II, p.366]{Jan}):

\[
R^{j}{f_{r}}_{*}{\mathcal L}(w,V) = {\mathcal L}({ws_{{i_r}}}, 
H^{j}(P_{\alpha_{i_r}}/B, {\mathcal L}(s_{\alpha_{i_r}},V))).
 \hspace{1.5cm}(Iso) \]

We use the following {\em ascending 1-step construction} as a basic tool 
in computing cohomology modules.

For $w \in W,$ let $l(w)$ denote the length of $w$. Let $\gamma$ be a 
simple root such that $l(w) = l(s_{\gamma}w) +1$. Let $Z(w, \underline i)$
be a BSDH-variety corresponding to a reduced
expression $w=s_{{i_{1}}}s_{{i_{2}}}\cdots s_{{i_{r}}}$,
where $\alpha_{i_{1}}=\gamma$. Then, we have an induced morphism

\[
g: Z(w, \underline i) \longrightarrow P_{\gamma}/B \simeq {\mathbb P}^1,
\]
with fibres $Z(s_{\gamma}w, \underline i')$, where $\underline i'=(i_2,i_3, \ldots, i_r)$.

By an application of the Leray spectral sequence together with the fact
that the base is ${\mathbb P}^1$, we obtain for every $B$-module $V$, the 
following exact sequence of $P_{\gamma}$-modules:
$$
0 \to H^{1}(P_{\gamma}/B, R^{j-1}{g}_{*}{\mathcal L}(w,V)) 
\to 
H^{j}(Z(w, \underline i) , {\mathcal L}(w,V)) \to
H^{0}(P_{\gamma}/B , R^{j}{g}_{*}{\mathcal L}(w,V) ) \to 0.
$$

Since for any $B$-module $V$, the vector bundle ${\mathcal L}(w,V)$
on $Z(w, \underline i)$ is the pull back of the homogeneous vector 
bundle from $X(w)$, we conclude that the cohomology modules
$$H^{j}(Z(w, \underline i) ,~{\mathcal L}(w,V))\cong H^{j}(X(w),
~{\mathcal L}(w,V))$$ (see  \cite[Theorem 3.3.4 (b)]{Bri}), and are independent
 of the choice of the reduced expression $\underline i$. Hence we denote
$H^{j}(Z(w, \underline i) ,~{\mathcal L}(w,V))$ by $H^j(w,V)$. 
For a character $\lambda$ of $B$, we denote the one-dimensional $B$-module corresponding to $\lambda$
by $\mathbb C_{\lambda}$. Further, we denote the 
cohomology modules $H^{j}(Z(w, \underline i) ,~{\mathcal L}(w, \mathbb C_{\lambda} ))$ by $H^j(w, \lambda)$. 

Rewriting the above short exact sequence using these simple notation, we have the following short 
exact sequence:
$$
0 \to H^{1}(s_{\gamma}, H^{j-1}(s_{\gamma}w, V)) 
\to 
H^{j}(w, V) \to
H^{0}(s_{\gamma} , H^{j}(s_{\gamma}w, V) ) \to 0.
$$

In this paper, the $B$-modules $V$ we deal with satisfy $R^{j}{g}_{*}{\mathcal L}(w,V)=0$
for all $j\geq 2$.
Moreover, we use only the following two special cases of the above short exact sequence, 
which we denote 
by {\it SES}. 
\begin{enumerate}
 \item $H^0(w, V)\simeq H^0(s_{\gamma}, H^0(s_{\gamma}w, V))$ for $j=0$.
 \item  $ 0 \to H^{1}(s_{\gamma}, H^{0}(s_{\gamma}w, V))
\to H^{1}(w, V) \to H^{0}(s_{\gamma} , H^{1}(s_{\gamma}w, V) ) \to 0$ for $j=1$.
\end{enumerate}

  Now, we recall the following 
result
 due to Demazure (\cite{Dem}, Page 1) on a short exact sequence of $B$-modules:

\begin{lemma}\label{dem1} Let $\alpha$ be a simple root and $\lambda \in X(T)$ be such that 
$\langle \lambda , \alpha \rangle  \geq 0$.
Let $ev$ denote the evaluation
map $H^0({s_\alpha, \lambda}) 
\longrightarrow \mathbb{C}_\lambda$. Then we have

\begin{enumerate}

\item If $\langle \lambda,\alpha \rangle=0$, then $H^0({s_\alpha, \lambda})\simeq\mathbb{C}_\lambda$.
\item If  $\langle \lambda,\alpha\rangle \geq 1$, then $\mathbb C_{s_{\alpha}(\lambda)}\hookrightarrow H^0({s_\alpha, \lambda})$
and there is a short exact sequence of $B$-modules:
$$ 0\longrightarrow  H^0({s_\alpha,\lambda-\alpha})\longrightarrow H^0({s_\alpha, \lambda})/ \mathbb C_{s_{\alpha}(\lambda)}\overset{ev}{\longrightarrow} \mathbb C_{\lambda}\longrightarrow 0.$$
Further more, $H^0({s_\alpha,\lambda-\alpha})=0$ when $\langle\lambda,\alpha \rangle=1$.
\item Let $n=\langle \lambda, \alpha \rangle$. 
As a $B$-module, $H^0(s_{\alpha}, \lambda)$ has a composition series
$$0\subsetneq V_n\subsetneq V_{n-1}\subsetneq \ldots \subsetneq V_0=H^0(s_{\alpha}, \lambda)$$ 
such that  $V_i/V_{i+1}\simeq \mathbb C_{\lambda -i\alpha}$
for $i=0,1, \cdots, n-1$ and  $V_n=\mathbb C_{s_{\alpha}(\lambda)}$.

\end{enumerate}

\end{lemma}

 We define the {\rm dot} action by $w\cdot\lambda=w(\lambda+\rho)-\rho$ for any $w \in W$ and $\lambda \in X(T)\otimes \mathbb R$.
 Note that $s_{\alpha}\cdot 0=-\alpha$ for $\alpha \in S$.
 As a consequence of the exact sequences of Lemma \ref{dem1}, we can prove 
the following.

Let $w\in W$, $\alpha$ be a simple root, and set $v=ws_{\alpha}$.
\begin{lemma} \label{lemma2.2}
If $l(w) = l(v)+1$, then, we have
\begin{enumerate}
\item  If $\langle \lambda , \alpha \rangle \geq 0$, then 
$H^{j}(w , \lambda) = H^{j}(v, H^0({s_\alpha, \lambda}) )$ for all $j\geq 0$.
\item  If $\langle \lambda ,\alpha \rangle \geq 0$, then $H^{j}(w , \lambda ) =
H^{j+1}(w , s_{\alpha}\cdot \lambda)$ for all $j\geq 0$.
\item If $\langle \lambda , \alpha \rangle  \leq -2$, then $H^{j+1}(w , \lambda ) =
H^{j}(w ,s_{\alpha}\cdot \lambda)$ for all $j\geq 0$. 
\item If $\langle \lambda , \alpha \rangle  = -1$, then $H^{j}( w ,\lambda)$ 
vanishes for every $j\geq 0$.
\end{enumerate}
\end{lemma}
\begin{proof}
 Choose a reduced expression of $w=s_{{i_1}}s_{{i_2}}\cdots s_{{i_r}}$
with $\alpha_{i_r}=\alpha$. Hence $v=s_{{i_1}}s_{{i_2}}\cdots 
s_{{i_{r-1}}}$ is a reduced 
expression for $v$.
Let $\underline i=(i_1, i_2, \ldots, i_r)$ and $\underline i'=(i_1, i_2, \ldots, i_{r-1})$.
Now consider the morphism $f_r : Z(w, \underline i) \longrightarrow Z(v,\underline i')$ 
defined as above.

Proof of (1): Since $\langle \lambda , \alpha \rangle  \geq 0$, we have 
$H^j(s_{\alpha}, \lambda)=0$ for every $j>0$.
Hence using the isomorphism ($Iso$), we have $R^{j}{f_{r}}_{*}{\mathcal L}(w, \lambda)=0$ for 
every $j>0$.
Therefore, 
by \cite[p.252, III, Ex $8.1$]{Hart} we have $H^i(w, \lambda)= H^i(v, H^0(s_{\alpha}, \lambda))$
for every $i\geq 0$.

Proof of (3): Since $\langle \lambda , \alpha \rangle  \leq -2$, by using (Borel-Weil-Bott
theorem) \cite[Theorem 2 (c)]{Dem} for $L_{\alpha}/B_{\alpha}(\simeq P_{\alpha}/B)$; we have 
$H^i(s_{\alpha}, \lambda)=0$ for $i\neq 1$ and $H^1(s_{\alpha}, \lambda)= H^0(s_{\alpha},
s_{\alpha}\cdot \lambda)$.
By ($Iso$), we have $R^{j}{f_{r}}_{*}{\mathcal L}(w, \lambda)=0$ for every $j\neq 1$. Hence by 
using Leray spectral sequence, we see that $H^{j+1}(w, \lambda)=
H^j(v, R^1{f_{r}}_{*}{\mathcal L}(w, \lambda))
=H^j(v, H^1(s_{\alpha}, \lambda))$ (see \cite[p.152, Section 5.8.6]{Weibel}). Hence 
$H^{j+1}(w, \lambda)=H^j(v, H^0(s_{\alpha}, s_{\alpha}\cdot \lambda))$ for every $j\geq 0$.
Since $\langle s_{\alpha}\cdot \lambda,  \alpha \rangle \geq 0$, by (1) we have 
$H^j(v, H^0(s_{\alpha}, s_{\alpha}\cdot \lambda))=H^j(w, s_{\alpha}\cdot \lambda)$ 
for every $j\geq 0$.
Hence we have $H^{j+1}(w, \lambda)= H^j(w, s_{\alpha}\cdot \lambda)$ for every $j\geq 0$. 

  Proof of (2): It follows from (3) by interchanging the role of $\lambda$ and 
  $s_{\alpha}\cdot \lambda$, because $\langle s_{\alpha} \cdot \lambda, \alpha \rangle = 
  -\langle \lambda, \alpha \rangle -2$.
  
 Proof of (4):  If $\langle \lambda , \alpha \rangle =-1$,  then $H^i(s_{\alpha}, \lambda)=0$ 
 for every $i\geq 0$ (see \cite[p.218, Proposition 5.2(b)]{Jan}).
  Now the proof of (4) follows by using similar arguments as in (1) and (3).  
\end{proof}

The following consequence of Lemma \ref{lemma2.2} will be used 
to compute cohomology modules in this paper.

Let $\pi:\widetilde{G}\longrightarrow G$ be the simply connected covering of $G$.
Let $\widetilde{L}_{\alpha}$  ( respectively,  $\widetilde{B}_{\alpha}$  
be the inverse image of $L_{\alpha}$ ( respectively, $B_{\alpha}$ ) in $\widetilde{G}$
under $\pi$. 

\begin{lemma}\label{lemma2.3}
Let $V$ be an irreducible  $\widetilde{L}_{\alpha}$-module. Let $\lambda$
be a character of $\widetilde{B}_{\alpha}$. Then, we have 
\begin{enumerate}
\item As $\widetilde{L}_{\alpha}$-modules, 
$H^j(\widetilde{L}_{\alpha}/\widetilde{B}_{\alpha}, V \otimes \mathbb C_{\lambda})\simeq V \otimes
   H^j(\widetilde{L}_{\alpha}/\widetilde{B}_{\alpha}, \mathbb C_{\lambda})$ for every $j\geq 0$.
\item If
$\langle \lambda , \alpha \rangle \geq 0$,  
$H^{j}(\widetilde{L}_{\alpha}/\widetilde{B}_{\alpha} , V\otimes \mathbb{C}_{\lambda}) =0$ 
for every $j\geq 1$.
\item If
$\langle \lambda , \alpha \rangle  \leq -2$, 
$H^{0}(\widetilde{L}_{\alpha}/\widetilde{B}_{\alpha} , V\otimes \mathbb{C}_{\lambda})=0$, 
and $$H^{1}(\widetilde{L}_{\alpha}/\widetilde{B}_{\alpha} , V\otimes \mathbb{C}_{\lambda})\simeq V \otimes H^{0}(\widetilde{L}_{\alpha}/\widetilde{B}_{\alpha} , 
\mathbb{C}_{s_{\alpha}\cdot\lambda}).$$ 
\item If $\langle \lambda , \alpha \rangle  = -1$, then 
$H^{j}( \widetilde{L}_{\alpha}/\widetilde{B}_{\alpha} , V\otimes \mathbb{C}_{\lambda}) =0$ 
for every $j\geq 0$.
\end{enumerate}
\end{lemma}
\begin{proof} Proof (1).
    By \cite[p.53, I, Proposition 4.8]{Jan} and \cite[p.77, I, Proposition 5.12]{Jan}, 
   for all $j\geq 0$, we have the following isomorphism of  $\widetilde{L}_{\alpha}$-modules:
   $$H^j(\widetilde{L}_{\alpha}/\widetilde{B}_{\alpha}, V \otimes \mathbb C_{\lambda})\simeq V \otimes
   H^j(\widetilde{L}_{\alpha}/\widetilde{B}_{\alpha}, \mathbb C_{\lambda}).$$ 
   
 Proof of (2), (3) and (4) follows from Lemma \ref{lemma2.2}  by taking $w=s_{\alpha}$ and 
 the fact that $\widetilde{L}_{\alpha}/\widetilde{B}_{\alpha} \simeq P_{\alpha}/B$.
 \end{proof}

Recall the structure of indecomposable 
$B_{\alpha}$-modules (see \cite[ p.130, Corollary 9.1]{Ka1}).

\begin{lemma}\label{lemma2.4}{\ }
\begin{enumerate}
\item
Any finite dimensional indecomposable $\widetilde{B}_{\alpha}$-module $V$ is isomorphic to 
$V^{\prime}\otimes \mathbb{C}_{\lambda}$ for some irreducible representation
$V^{\prime}$ of $\widetilde{L}_{\alpha}$ and for some character $\lambda$ of $\widetilde{B}_{\alpha}$.
\item
Any finite dimensional indecomposable $B_{\alpha}$-module $V$ is isomorphic to 
$V^{\prime}\otimes \mathbb{C}_{\lambda}$ for some irreducible representation
$V^{\prime}$ of $\widetilde{L}_{\alpha}$ and for some character $\lambda$ of $\widetilde{B}_{\alpha}$.
\end{enumerate}
\end{lemma}
\begin{proof}
Proof of (1) follows from \cite[p.130, Corollary 9.1]{Ka1}.

Proof of (2) follows from the fact that every $B_{\alpha}$-module can be viewed as a $\widetilde{B}_{\alpha}$-module via the natural homomorphism.
\end{proof}

Now, we prove the following: 
\begin{corollary}\label{commuting}
 Let $w=s_{i_1}s_{i_2}\cdots s_{i_r}$ be a reduced expression for $w$ such that
 $\langle \alpha_{i_j}, \alpha_{i_r} \rangle =0$ for every $j=1,2, \ldots, r-1$.
 Then, $H^0(w, \alpha_{i_r})$ is isomorphic to $H^0(s_{i_r}, \alpha_{i_r})
 (\simeq sl_{2, \alpha_{i_r}})$.
\end{corollary}
\begin{proof}
Since $L_{\alpha_{i_r}}/B_{\alpha_{i_r}}\hookrightarrow P_{\alpha_{i_r}}/B$ is an isomorphism, we have 
  $$sl_{2, \alpha_{i_r}} \simeq H^0(L_{\alpha_{i_r}}/B_{\alpha_{i_r}}, \alpha_{i_r}) \simeq 
  H^0(s_{i_r}, \alpha_{i_r}).$$
  We note that $sl_{2, \alpha_{i_r}}$ gets a natural $B$-module structure via the above 
  isomorphism $sl_{2, \alpha_{i_r}} \simeq H^0(s_{i_r}, \alpha_{i_r})$.

Let $v=s_{i_1}s_{i_2}\cdots s_{i_{r-1}}$. If $l(v)=0$, then $w=s_{i_r}$ and  we are done. 
Otherwise, 
 let $v^{'}=s_{i_2}\cdots s_{i_{r-1}}$.
 By induction on $l(v)$, we have $$H^0(s_{i_2}\cdots s_{i_r}, \alpha_{i_r})=H^0(s_{i_r}, 
 \alpha_{i_r}).$$
By {\it SES}, we have 
$$H^0(w, \alpha_{i_r})=H^0(s_{i_1}, H^0(s_{i_2}\ldots s_{i_r}, \alpha_{i_r}))=H^0(s_{i_1},
H^0(s_{i_r},\alpha_{i_r})).$$

Since $\langle \alpha_{i_r}, \alpha_{i_1} \rangle =0$ and 
$\langle -\alpha_{i_r}, \alpha_{i_1} \rangle =0$, by Lemma \ref{lemma2.4}, 
$H^0(s_{i_r}, \alpha_{i_r})$
is the trivial $B_{\alpha_{i_1}}$-module of dimension $3$. Hence, the vector bundle
$\mathcal L (s_{i_1}, H^0(s_{i_r}, \alpha_{i_r}))$
on $X(s_{i_1})\simeq \mathbb P^1$ is the trivial bundle of rank $3$. 
Thus, we have $$ H^0(s_{i_1}, H^0(s_{i_r},\alpha_{i_r}))=H^0(s_{i_r}, \alpha_{i_r}).$$\end{proof}

We recall the following vanishing results  from \cite{Ka4} (see  \cite[ Corollary 3.6]{Ka4}
and \cite[ Corollary 4.10]{Ka4}).

\begin{lemma}\label{vanishing}
Let $w\in W$, and $\alpha \in R^+$. Then, we have
  \begin{enumerate}
   \item $H^j(w, \alpha)=0$ for all $j\geq 2$.
   \item If $G$ is simply laced, $H^j(w, \alpha)=0$ for all $j\geq 1$.
  \end{enumerate}
 \end{lemma}

  Let $T_{G/B}$ denote the tangent bundle of the 
flag variety $G/B$. By abuse of notation, we denote the restriction $T_{G/B}$ to $X(w)$ by 
$T_{G/B}$. As we discussed in the introduction about the condition $w^{-1}(\alpha_0)<0$, 
we state the following theorem from \cite{Ka4} (see \cite[Theorem 3.7, Theorem 3.8 and 
Theorem 4.11]{Ka4}).
 
\begin{theorem}\label{AS}
Let $w \in W$. Then
\begin{enumerate}

\item
$H^{i}(X(w), T_{G/B})=0$ for every $i\geq 1$. 
\item
The adjoint representation $\mathfrak{g}$ of $G$ is a $B$-submodule of $H^{0}(X(w) , T_{G/B})$ 
if and only if $w^{-1}(\alpha_{0})<0$.
\item If $G$ is simply laced, $H^{0}(X(w) , T_{G/B})$ is the adjoint representation 
$\mathfrak{g}$ of  $G$ if and only if $w^{-1}(\alpha_{0})<0$. 
\item Assume that $G$ is simply laced and $X(w)$ is a smooth Schubert variety. 
Let $Aut^0(X(w))$ be the connected component 
of the automorphism group of $X(w)$ containing the identity automorphism.
Let $P_w$ denote the stabilizer of $X(w)$ in $G$. Let $\phi_w: P_w\longrightarrow Aut^0(X(w))$ be the homomorphism 
induced by the action of $P_w$ on $X(w)$. 
Then, we have \\(i) $\phi_w: P_w\longrightarrow Aut^0(X(w))$ is surjective.
\\(ii) $\phi_w: P_w\longrightarrow Aut^0(X(w))$ is an isomorphism if and only if $w^{-1}(\alpha_0)<0$.
\end{enumerate}
\end{theorem}

\section{Vanishing of the Higher Cohomology of the Tangent Bundle of $Z(w, \underline i)$:}

In this section, we prove that a BSDH variety has unobstructed 
deformations and it has no deformations whenever the group $G$ is simply 
laced. 

We recall that the BSDH-variety corresponding to 
 a reduced expression $w=s_{{i_1}}s_{{i_2}}\cdots s_{{i_r}}$ is denoted by 
 $Z(w, \underline i)$ 
 and we denote the tangent bundle of $Z(w,\underline i)$ by $T_{(w,\underline i)}$, where 
 $\underline i=(i_1, i_2, \ldots, i_r)$.

Let $w =s_{{i_1}}s_{{i_2}}\cdots s_{{i_r}}$, 
$v=s_{{i_1}}s_{{i_2}}\cdots 
s_{{i_{r-1}}}$ and $\underline i'=(i_1,i_2,\ldots, i_{r-1})$. Note that $l(v)= l(w)-1$.
Consider the fibration $f_r:Z(w, \underline i) \longrightarrow Z(v,\underline i')$
as in Section $2$. 
One can easily see that this fibration is the fibre product of 
$\pi_r:G/B\to G/P_{\alpha_{i_{r}}}$ and 
$\pi_r\circ\phi_{v}:Z(v, \underline {i'}) \to G/P_{\alpha_{i_{r}}}$;
namely, we have the following commutative diagram :

\begin{center}  
$\xymatrix{ Z(v, \underline i')\times_{ G/P_{\alpha_{i_r}}} G/B =Z(w, \underline i)
\ar[d]_{f_r}\ar[rrr]^{\phi_w} &&&  G/B \ar[d]_{\pi_r}\\ 
Z(v, \underline i') \ar[rrr]^{\pi_r\circ\phi_{v}} &&& G/P_{\alpha_{i_r}}}$ 
\end{center}

The relative tangent bundle of $\pi_r$ is the line bundle $\mathcal{L}(w_0, \alpha_{i_r})$. 
Hence the relative tangent
bundle of $f_r$ is $\phi_w^*\mathcal{L}(w_0, \alpha_{i_r})$. 
By taking the differentials of this smooth fibration $f_{r}$ we obtain the following exact 
sequence:
$$ 0\to \phi_w^*\mathcal{L}(w_0, \alpha_{i_r}) \to T_{(w, \underline i)} \to 
f_r^*T_{(v, \underline i')} \to 0 . \hspace{1cm} (rel)$$

We use the above short exact sequence $(rel)$ and Lemma \ref{vanishing} to prove the following:

\begin{proposition}\label{prop3.1}
 Let $w\in W$, $w =s_{{i_1}}s_{{i_2}}\cdots s_{{i_r}}$ 
be a reduced expression for $w$ and let $\underline i=(i_1, i_2, \ldots, i_r)$. Then, we have
 \begin{enumerate}
  \item $H^j(Z(w, \underline i), T_{(w, \underline i)})=0$ for all $j\geq 2$.
  \item If $G$ is simply laced, $H^j(Z(w, \underline i), T_{(w, \underline i)})=0$ for all
  $j\geq 1$.
 \end{enumerate}
\end{proposition}
\begin{proof} We start by  proving (2).  We first recall the following isomorphism 
(see \cite[ Theorem $3.3.4 (b)$ ]{Bri}):
$$H^j(Z(w, \underline i), \phi_w^*\mathcal L(w_0, \alpha_{i_r}))\simeq H^j(X(w), 
\mathcal L(w, \alpha_{i_r}))=H^j(w, 
\alpha_{i_r}) ~~ for~ all~~ j\geq 0.$$

Let $v=s_{{i_1}}s_{{i_2}}\cdots s_{{i_{r-1}}}$ and
$\underline i'=(i_1, i_2, \ldots, i_{r-1})$. 
Since $f_r:Z(w, \underline i) \longrightarrow Z(v,\underline i')$ is a 
smooth fibration with fibre $\mathbb P^1$, by using \cite[p.288, Corollary $12.9$]{Hart} and  
\cite[p.369, Section $14.6 (3)$ ]{Jan} 
we have $H^j(Z(w, \underline i), f_r^*T_{(v, \underline i')})=
H^j(Z(v, \underline i'), T_{(v, \underline i')})$ for every $j\geq 0$. 

By considering the long exact sequence associated to the short exact sequence $(rel)$ and using 
above arguments, we have the following long exact sequence of $B$-modules:\\ 
 $0\longrightarrow H^0(w, \alpha_{i_r})\longrightarrow H^0(Z(w, \underline i), 
 T_{(w,\underline i)})\longrightarrow H^0(Z(v, \underline i'), T_{(v, \underline i')})
 \longrightarrow H^1(w, \alpha_{i_r})\longrightarrow \\
 H^1(Z(w, \underline i), T_{(w,\underline i)})\longrightarrow H^1(Z(v, \underline i), 
 T_{(v,\underline i')})
 \longrightarrow H^2(w, \alpha_{i_r})\longrightarrow H^2(Z(w, \underline i), 
 T_{(w,\underline i)}) \longrightarrow \\ H^2(Z(v, \underline i'), T_{(v, \underline i')})
 \longrightarrow H^3(w, \alpha_{i_r})\longrightarrow \cdots$.\\
  Since $G$ is simply laced, by Lemma \ref{vanishing} (2), we have $H^j(w, \alpha_{i_r})=0$
  for every $j\geq 1$. Thus we have $H^j(Z(w, \underline i), T_{(w,\underline i)})=
  H^j(Z(v, \underline i'), T_{(v,\underline i')})$ for every $j \geq 1$. Now the proof follows
  by induction on $l(w)$.

Proof of (1) is similar by using Lemma \ref{vanishing} (1).
\end{proof}

Note: The long exact sequence associated to the short exact sequence $(rel)$ which is considered 
in the proof of the Proposition \ref{prop3.1} will be used
 frequently in the future.
 We call this $\it{LES}$.
 
 Proposition \ref{prop3.1}(1) yields
$H^2(Z(w, \underline i), T_{(w, \underline i)})=0$. Hence, 
we see that $Z(w, \underline i)$ has unobstructed deformations. 
That is, $Z(w, \underline i)$ admits a smooth versal deformation 
(see \cite [p.273, lines 19-21]{Huy}). 

If in addition $G$ is simply laced, 
Proposition \ref{prop3.1}(2) yields 
$H^1(Z(w, \underline i), T_{(w, \underline i)})=0$. Using \cite [p. 272, Proposition 6.2.10]{Huy},
we see that $Z(w, \underline i)$ has no deformations. That is, a BSDH variety for a simply laced 
group $G$ is rigid.

\section{Cohomology of the relative tangent bundle on $Z(w, \underline i)$}

In this section, we compute the cohomology groups of the 
relative tangent bundle on $Z(w, \underline i)$.

We use the notation as in the previous section. 
For a $B$-module $V$ and a character $\mu\in X(T)$, we denote by $V_{\mu}$, the 
weight space for the action of $T$. By the definition, it is the space of all 
vectors $v$ in 
$V$ such that, for all $t\in T$, $t\cdot v=\mu(t)v$.  We denote by $dim(V_{\mu})$ the dimension of 
the space $V_{\mu}$.  

Given a weight $\lambda \in X(T)$ and a simple root $\gamma \in S$ such that $\langle \lambda, 
\gamma \rangle \geq 0$,
we  recall that the $\gamma$-string of $\lambda$ 
is the set $\{\lambda, \lambda-\gamma, \cdots, \lambda-\langle \lambda, \gamma 
\rangle \gamma\}$ of weights, which by Lemma \ref{dem1}, is the set of weights occuring in $H^0(s_{\gamma}, \lambda)$.

Recall, the partial order $\leq$ on $X(T)$ given by $\mu\leq \lambda$
if $\lambda-\mu$ is a non-negative integral linear combination of simple 
roots. We say that $\mu < \lambda$ if in addition $\lambda-\mu$ is non zero.

We begin by proving the following Lemma:

 Let $R_s$ (respectively, $ R^-_s$) be the set of short roots (respectively, negative short roots).
 \begin{lemma}\label{l1}
 Let $w\in W$, $V$ be a $B$-module. Then we have 
 \begin{enumerate} 
 \item
  If there is a character $\lambda_0\in X(T)$ such that  $V_{\mu}=0$ unless $\mu \leq \lambda_0$ 
  (respectively,  $\mu< \lambda_0$), 
 then $H^0(w, V)_{\mu}=0$ unless $\mu \leq \lambda_0$ 
  (respectively,  $\mu< \lambda_0$).
 \item If $V_{\mu}=0$ for every $\mu\in X(T)\setminus (R\cup \{0\})$, then $H^0(w, V)_{\mu}=0$ 
 for every $\mu\in X(T)\setminus (R\cup \{0\})$.
 \item If $V_{\mu}=0$ for every $\mu \in X(T)\setminus (R_s \cup \{0\})$, then $H^0(w, V)_{\mu}=0$
 for every $\mu \in X(T)\setminus (R_s \cup \{0\})$.
 \item If $V_{\mu}=0$ for every $\mu\in X(T)\setminus (R^-_{s}\cup \{0\})$, 
 then $H^0(w, V)_{\mu}=0$ for every $\mu\in X(T)\setminus (R^-_s \cup \{0\})$. 
 \end{enumerate}
\end{lemma}
\begin{proof}Proof of (1):
Let $V$ be a $B$-module and $\lambda_0\in X(T)$ such that  $V_{\mu}=0$ if $\mu\nleq \lambda_0$.
Proof is by induction on $l(w)$.
 If $l(w)=0$ there is nothing to prove. Otherwise, we can choose a $\gamma \in S$ such that 
 $l(s_{\gamma}w)=l(w)-1$.
 Let $u=s_{\gamma}w$. 
 By {\it{SES}}, the $B$-modules $H^0(s_{\gamma}, H^0(u, V))$ and $H^0(w, V)$ are isomorphic. 
 
 Let $\mu\in X(T)$ be a weight of $H^0(w, V)$ (i.e, $H^0(w, V)_{\mu}\neq 0$). Then there is an 
 indecomposable $B_{\gamma}$-summand $V'$ of 
 $H^0(u, V)$ such that $H^0(s_{\gamma}, V')_{\mu}\neq 0$. By Lemma \ref{lemma2.4}, we have 
 $V'=V''\otimes \mathbb C_{\mu'}$ for some irreducible $\widetilde{L}_{\gamma}$-module $V''$
  and for some character $\mu'$ of $\widetilde{B}_{\gamma}$. By Lemma \ref{lemma2.3}, 
  we have $H^0(s_{\gamma}, V')=V''\otimes H^0(s_{\gamma},\mu')$ and $\langle \mu', 
  \gamma \rangle \geq 0$.
  Now, let $\mu''$ be the highest weight of $V''$. Then, $H^0(s_{\gamma}, V')=H^0(s_{\gamma}, 
  \mu'')\otimes H^0(s_{\gamma}, \mu')$.
 By  the description of the weights of $H^0(s_{\gamma}, 
  \mu'')\otimes H^0(s_{\gamma}, \mu')$, any weight $\lambda$ of 
 $H^0(s_{\gamma}, V')$ is of the form  $\lambda=\mu_1+\mu_2$
where $\mu_1=\mu''-a_1\gamma$ and $\mu_2=\mu'-a_2\gamma$ 
for some integers $0\leq a_1 \leq  \langle \mu'', \gamma \rangle$, $0\leq a_2 \leq  \langle \mu', 
\gamma \rangle$.
Thus, we have  $\lambda=\mu''+\mu'-(a_1+a_2)\gamma$.
  
  Hence, any weight $\lambda$ of $H^0(s_{\gamma}, V')$ satisfies $\lambda\leq \mu'+\mu''$.
  In particular, $\mu\leq \mu'+\mu''$.
  Note that since $\mu'+\mu''$ is the highest weight of $H^0(s_{\gamma}, V')$, $H^0(u, V)_
  {\mu'+\mu''}\neq 0$. 
  By induction on $l(w)$, $\mu'+\mu''\leq \lambda_0$. 
  Hence, we have $\mu\leq \lambda_0$.
  
  Proof of $V_{\mu}=0$ unless  $\mu< \lambda_0 \Longrightarrow H^0(w, V)_{\mu}=0$ unless  $\mu< \lambda_0$ is 
similar.
  
  Proof of (2): Assume that $H^0(w, V)_{\mu}\neq 0$. We use the same notation as 
  in the proof of (1).
  We have $H^0(s_{\gamma}, V')=H^0(s_{\gamma}, \mu')\otimes H^0(s_{\gamma}, \mu'')$.
  Since $V'_{\mu'+\mu''}\neq 0$, by induction on $l(w)$,   $\mu'+\mu''\in R\cup \{0\}$. 
    By the proof of (1), the  weights of $H^0(s_{\gamma}, V')$ are of the form 
  $\mu=\mu'+\mu''-j \gamma$ for some integer $0\leq j \leq \langle \mu'+\mu'', \gamma \rangle$. 
  If $\mu'+\mu''=0$, then $j=0$ and so $\mu=\mu'+\mu''=0$. Otherwise, $\mu'+\mu''$ is a root, it follows that
  $\mu$ is a root (see \cite[p.45, Section 9.4]{Hum1}).
  
  Proof of (3) follows from the proof of (2) because any root in the $\gamma$-string of 
  a short root is short.
  
  Proof of (4) follows from (1) (by taking $\lambda_0=0$) and (3).
  \end{proof}

 \begin{lemma}\label{l2} Let $w\in W$. Then we have,
 $H^1(w, \mathfrak{b})_{\mu}= 0$ unless $\mu$ is a negative short root.
\end{lemma} 
\begin{proof}
 If $l(w)=0$, we are done. Otherwise, choose $\gamma \in S$ such that $l(s_{\gamma}w)=l(w)-1$.
Let $u=s_{\gamma}w$. Then by {\it SES},  we have 
$$0\longrightarrow H^1(s_{\gamma}, H^0(u, \mathfrak b))\longrightarrow H^1(w, \mathfrak b)
\longrightarrow H^0(s_{\gamma}, H^1(u, \mathfrak b))\longrightarrow 0$$
By induction on $l(w)$, $H^1(u, \mathfrak b)_{\mu}=0$ unless $\mu$ is a negative short root.
By Lemma \ref{l1} (4), $H^0(s_{\gamma}, H^1(u, \mathfrak b))_{\mu}=0$ unless $\mu$ is a negative
short root.

Now, we prove that $H^1(s_{\gamma}, H^0(u, \mathfrak b))_{\mu}=0$ unless $\mu$ is a negative
short root.
   Assume that $H^1(s_{\gamma}, H^0(u, \mathfrak b))_{\mu}\neq 0$. Then there exists an
   indecomposable $B_{\gamma}$-direct summand $V_1$ of $H^0(u, \mathfrak b)$ such that 
   $H^1(s_{\gamma}, V_1)_{\mu}\neq 0$.
   By Lemma \ref{lemma2.4}, $V_1=V'\otimes \mathbb C_{\mu'}$ for some irreducible 
   $\widetilde{L}_{\gamma}$-module $V'$ and for some character $\mu'$ of $\widetilde{B}_{\gamma}$.
  Since $H^1(s_{\gamma}, V_1)\neq 0$, by Lemma \ref{lemma2.3} we have $\langle \mu', 
  \gamma \rangle \leq -2$ and
  $H^1(s_{\gamma}, V_1)=V'\otimes H^0(s_{\gamma}, s_{\gamma} \cdot \mu')$. Then any
  weight $\mu''$ of $H^1(s_{\gamma}, V_1)$ is in the $\gamma$-string from $\mu_1+\gamma=
  \mu_1+\rho-s_{\gamma}(\rho)=s_{\gamma}(s_{\gamma}\cdot \mu_1)$ to 
  $s_{\gamma}\cdot \mu_1$, where $\mu_1$ is the lowest weight of $V_1$.
  
  Note that by \cite[Lemma 2.6]{Ka4}, the evaluation map $ev: H^0(u, \mathfrak b)\longrightarrow
  \mathfrak b$ is injective. Hence, if $H^0(u, \mathfrak b)_{-\gamma}\neq 0$
 then $\mathbb C. h_{\gamma}\oplus \mathbb C_{-\gamma}$ is an indecomposable $B_{\gamma}$-direct 
 summand of $H^0(u, \mathfrak b)$ (here
 $h_{\gamma}$ is a basis vector of the zero weight space of $sl_{2, \gamma}$).
  By Lemma \ref{lemma2.4}, we have $$\mathbb C.h_{\gamma}\oplus\mathbb C_{-\gamma}=
  V\otimes\mathbb{C}_{-\omega_{\gamma}}$$ where $V$ is the standard $2$-
dimensional representation of $\widetilde{L}_{\gamma}.$
By Lemma \ref{lemma2.3}, we have $$H^0(s_{\gamma}, V\otimes\mathbb{C}_{-\omega_{\gamma}} )= 
V\otimes H^0(s_{\gamma}, -\omega_{\gamma}).$$
    Since $\langle -\omega_{\gamma}, \gamma \rangle =-1$, by Lemma \ref{lemma2.2}, 
    $H^1(s_{\gamma}, \mathbb C.h_{\gamma}\oplus \mathbb C_{-\gamma} )=0$.

  Since $V_1$ is a $B$-submodule of $\mathfrak b$ and $H^1(s_{\gamma}, V')\neq 0$, by the above 
  arguments, we see that
  $V_1$ is not isomorphic to $\mathbb C.h_{\gamma}\oplus\mathbb C_{-\gamma}$. In particular, 
  we have $\mu_1\in R^-\setminus \{-\gamma\}$. 
  Let $\lambda$ be the lowest weight of $V'$. Then, we have $\mu_1=\lambda+\mu'$. 
  Since $\langle \lambda, \gamma \rangle\leq 0$ and $\langle \mu', \gamma \rangle \leq -2$, 
  we have   $\langle \mu_1, \gamma \rangle \leq -2$. Further by \cite[p.45, Section 9.4]{Hum1},
  we have $-3\leq \langle \mu_1, \gamma \rangle $.
  Then, the $\gamma$-string of $\mu$ is either $\mu+\gamma$ (if $\langle \mu_1 , \gamma \rangle=-2$)
   or $\mu+\gamma, \mu+2\gamma$ (if $\langle \mu_1 , \gamma \rangle=-3$). In particular, any weight $\mu''$ of $H^1(s_{\gamma}, V_1)$ satisfies $|\langle \mu'', \gamma \rangle |
  \leq 1 $ and  $\mu''$ is a negative short root. In particular, $\mu$ is a negative short root.
  
  Hence by the above short exact sequence, we conclude that $H^1(w, \mathfrak b)_{\mu}= 0$ unless $\mu$ 
  is a negative short root.
\end{proof}

Recall from Section 2 that $h_{\alpha_{i_{r}}}$ is a basis vector of the zero weight space of 
 $sl_{2, \alpha_{i_{r}}}$.

 \begin{lemma}\label{1}  Let $w \in W$ and fix a reduced expression  
 $w=s_{{i_1}}s_{{i_2}}\cdots s_{{i_r}}$. Then,
 \begin{enumerate} 
 \item
  If there is a $1\leq j < r-1$ such that  $\alpha_{i_j}=\alpha_{i_r}$, then we have 
  $H^0(w, \alpha_{i_r})_0=0$.  
\item If $\alpha_{i_j}\neq \alpha_{i_r}$ for all $1\leq j <r-1$, then $\mathbb C.h_{\alpha_{i_r}}\oplus \mathbb C_{-\alpha_{i_r}}$
is a $B_{\alpha_{i_r}}$-submodule of $H^0(w, \alpha_{i_r})$, and $H^0(w, \alpha_{i_r})_0= \mathbb C.h_{\alpha_{i_r}}$.
In particular, $dim(H^0(w, \alpha_{i_r})_0)=1$.
 
 \end{enumerate}
 \end{lemma}
 \begin{proof} Proof of (1): If there is a $1\leq j < r-1$ such that $\alpha_{i_j}=\alpha_{i_r}$, without loss of 
 generality we may assume that
there is no $k$ such that $j<k<r-1$ and $\alpha_{i_k}=\alpha_{i_r}$. Since  $w=s_{i_1}s_
{i_2}\cdots s_{i_r}$ is a reduced expression, there exists a $j< j' \leq r-1$ 
such that $\langle\alpha_{i_r}, \alpha_{i_{j'}}\rangle\leq -1$ and 
$\langle\alpha_{i_r}, \alpha_{i_{k}}\rangle =0$ for every $k$ such that $j'<k<r$.  
  By Corollary \ref{commuting}, we have the following isomorphism of $B$-modules:
$$ H^0(s_{i_{j'+1}}\cdots s_{i_r}, \alpha_{i_r})\simeq H^0(s_{i_r},
\alpha_{i_r})\simeq sl_{2, \alpha_{i_r}}.$$

By {\it SES}, we have  $H^0(s_{i_{j'}}\cdots s_{i_r}, \alpha_{i_r}) \simeq  
H^0(s_{i_{j'}}, H^0(s_{i_{j'+1}}\cdots s_{i_r}, \alpha_{i_r}))$ as 
$B$-modules.
 
 Then,
 $$H^0(s_{i_{j'}}\cdots s_{i_r}, \alpha_{i_r}) \simeq H^0(s_{i_{j'}}, 
 H^0(s_{i_{r}},\alpha_{i_r}))\simeq \mathbb C.h_{\alpha_{i_r}}\oplus\mathbb
 C_{-\alpha_{i_r}} \oplus (\bigoplus_{m=1}^{-\langle \alpha_{i_r}, \alpha_{i_{j'}}\rangle}\mathbb 
 C_{-\alpha_{i_r}-m\alpha_{i_{j'}}}).$$
 
 Since $\langle \alpha_{i_r}, \alpha_{i_k} \rangle \leq 0$ for every $j+1\leq k \leq j'-1$, 
 we conclude that 
 the indecomposable $B_{\alpha_{i_r}}$-summand $\mathbb C.h_{\alpha_{i_r}}\oplus \mathbb 
 C_{-\alpha_{i_r}}$
 is in the image of the evaluation map
 $$ev: H^0(s_{i_{j+1}}\cdots s_{i_{j'-1}}, H^0(s_{i_{j'}}\cdots s_{i_r}, 
 \alpha_{i_r}))\longrightarrow H^0(s_{i_{j'}}\cdots s_{i_r}, \alpha_{i_r}).$$
 Since $H^0(s_{i_{j+1}}\cdots s_{i_r}, \alpha_{i_r})\simeq H^0(s_{i_{j+1}}
 \cdots s_{i_{j'-1}}, H^0(s_{i_{j'}}\cdots s_{i_r}, \alpha_{i_r}))$, 
 the indecomposable $B_{\alpha_{i_r}}$-module
 $\mathbb C.h_{\alpha_{i_r}}\oplus \mathbb C_{-\alpha_{i_r}}$ is a direct summand of 
 $H^0(s_{i_{j+1}}\cdots s_{i_r}, \alpha_{i_r})$.
  By similar arguments as in the proof of Lemma \ref{l2} and using Lemma \ref{lemma2.4}, 
  we have $H^0(s_{i_j}, \mathbb C.h_{\alpha_{i_r}}\oplus\mathbb C_{-\alpha_{i_r}})=0$.

 Now, let $u_1=s_{i_1}\cdots s_{i_{j-1}}$ and $u_{2}=s_{i_j}\cdots 
 s_{i_r}$.
 From the above arguments, we see that $H^0(u_2, \alpha_{i_r})_{\mu}=0$ unless $\mu<-\alpha_{i_r}$
 and $\mu\in R$.
 By  Lemma \ref{l1}, if $H^0(u_1, H^0(u_2, \alpha_{i_r}))_{\mu}\neq 0$ then $\mu<-\alpha_{i_r}$ 
 and $\mu \in R$.
Hence, the zero weight space of $H^0(w, \alpha_{i_r})$ is zero.

Proof of (2): Proof is similar to the proof of (1), for the completeness we will give the proof.\\
If $\langle \alpha_{i_j}, \alpha_{i_r} \rangle=0$ for every $1\leq j\leq r-1$, then by Corollary 
\ref{commuting}, we have $H^0(w, \alpha_{i_r})=sl_{2,\alpha_{i_r}}$. Hence, (2) holds in this case.

Otherwise, there exists $1\leq j\leq r-1$ such that $\langle \alpha_{i_j}, \alpha_{i_r} \rangle\neq 0$.
Let $1\leq k \leq r-1$ be the largest integer such that $\langle \alpha_{i_k}, \alpha_{i_r} \rangle\neq 0$.
Then by {\it SES} and Corollary \ref{commuting}, we have 
$$H^0(w, \alpha_{i_r})\simeq H^0(s_{i_1}s_{i_2}\cdots
s_{i_k}, H^0(s_{i_{k+1}}\cdots s_{i_{r}}, \alpha_{i_r}))\simeq H^0(s_{i_1}s_{i_2}\cdots
s_{i_k},sl_{2,\alpha_{i_r}}).$$
Since $\langle \alpha_{i_r}, \alpha_{i_k} \rangle \leq -1$, we have $$ H^0(w, \alpha_{i_r})\simeq H^0(s_{i_1}s_{i_2}\cdots
s_{i_{k-1}},  \mathbb C.h_{\alpha_{i_r}}\oplus 
 \bigoplus_{m=0}^{-\langle \alpha_{i_k}, \alpha_{i_r} \rangle}
\mathbb C_{-\alpha_{i_r}-m\alpha_{i_k}}).$$
Since $\alpha_{i_j}\neq \alpha_{i_r}$ for all $1\leq j <r-1$, we see that $\mathbb C.h_{\alpha_{i_r}} \oplus \mathbb C_{-\alpha_{i_r}}$ is an 
indecomposable $B_{\alpha_{i_r}}$-submodule of $H^0(w, \alpha_{i_r})$.
Further, $H^0(w, \alpha_{i_r})_0= \mathbb C.h_{\alpha_{i_r}}$ and so $dim(H^0(w, \alpha_{i_r})_0)=1$. This completes the proof of the lemma.
\end{proof}

Now onwards we denote by  $M_{\geq 0}$ the semi subgroup of $Hom_{\mathbb R}(\mathfrak 
h_{\mathbb R}, \mathbb R)$
generated by  the set $S$ of all simple roots .

\begin{lemma} 
\label{lemma1}
Let $w\in W$.
 Let $\mu\in  M_{\geq 0}\setminus \{0\}$ and let $\alpha \in S$. Then, we have 
 \begin{enumerate} 
 \item  If $H^0(w, \alpha)_{\alpha}\neq 0$, then $dim(H^0(w, \alpha)_{\alpha})=1$.
 
 \item
  $H^0(w, \alpha)_{\mu}\neq 0$ if and only if $\mu=\alpha$ and the evaluation map 
  $ev: H^0(w, \alpha) \longrightarrow \mathbb C_{\alpha} $ is surjective.
 \end{enumerate}
\end{lemma}
 \begin{proof} Proof of (1): 
 Let $w_1\in W$ be an element of minimal length such that $w_1(\alpha)$ is 
a dominant weight.
Note that if $l(w_1)=0$, then $\alpha$ is dominant. In particular, $G$ is of rank $1$ and 
$w\in \{id, s_{\alpha}\}$. Hence  $dim(H^0(w, \alpha)_{\alpha})=1$.
Otherwise, there exists a $\gamma \in S$ such that $l(w_1s_{\gamma})=l(w_1)-1$ and 
$\langle \alpha, \gamma \rangle < 0$. Hence by Lemma \ref{dem1}, $\mathbb C_{\alpha}$ is a 
$B$-submodule of $H^0(s_{\gamma}, s_{\gamma}(\alpha))$.
Then $H^0(w, \alpha)$ is a $B$-submodule of $H^0(w, H^0(s_{\gamma}, s_{\gamma}(\alpha)))$. 
Since $H^0(w, \alpha)_{\alpha}\neq 0$, by \cite[p.110, Theorem 3.3]{Ka1} (see also \cite{Deb} 
and \cite{Polo})
we have $l(ws_{\gamma})=l(w)+1$ (Note that since $\langle \alpha, \gamma \rangle < 0$, the 
regularity of $\lambda$ as in \cite[p.110, Theorem 3.3]{Ka1} does not play a role).
By  Lemma \ref{lemma2.2},
we have $$H^0(ws_{\gamma}, s_{\gamma}(\alpha))= H^0(w, H^0(s_{\gamma}, s_{\gamma}(\alpha))).$$
Hence $H^0(w, \alpha)$ is a $B$-submodule of $H^0(ws_{\gamma}, s_{\gamma}(\alpha))$.
By induction on $l(w_1)$,
$H^0(w, \alpha)$ is a $B$-submodule of $H^0(ww_1^{-1}, w_1(\alpha))$. 
Since $w_1(\alpha)$ is dominant, $H^0(ww_1^{-1}, w_1(\alpha))$
is a quotient of the $B$-module $H^0(w_0, w_1(\alpha))$. Further, since the multiplicity
of the weight $\alpha$ in $H^0(w_0, w_1(\alpha))$ is 1, 
 the multiplicity of the weight $\alpha$ in $H^0(ww_1^{-1}, w_1(\alpha))$ is at most 1. 
 Hence, we conclude that $dim(H^0(w, \alpha)_{\alpha})=1$.

Proof of (2): 

 Assume that $H^0(w, \alpha)_{\mu}\neq 0$.
If $l(w)=0$, there is nothing to prove.
Assume $l(w)>0$. Therefore, we can choose a $\gamma \in S$ such that $l(s_{\gamma}w)=l(w)-1$. 
Let $u=s_{\gamma}w$.
By {\it SES}, we have $H^0(w,\alpha)=H^0(s_{\gamma}, H^0(u, \alpha))$.

Since  $H^0(w, \alpha)_{\mu}\neq 0$, there exists an indecomposable $B_{\gamma}$-summand $V$ of 
$H^0(u, \alpha)$ such that $H^0(s_{\gamma}, V)_{\mu}\neq 0$.
Let $\mu'$ be the highest weight of $V$. By Lemma \ref{lemma2.4},
we have $V=V'\otimes \mathbb C_{\lambda}$ for some character $\lambda$ of $\widetilde{B}_{\gamma}$ and 
for some irreducible $\widetilde{L}_{\gamma}$-module $V'$. Let $\lambda_{1}$ be a highest weight of $V'$. By  
similar arguments as in the proof of Lemma \ref{l1}, we have  $\lambda_1+\lambda =\mu'$, 
and $\mu=\mu'-a\gamma$ where $0\leq a \leq \langle \mu', \gamma \rangle$.
Therefore, $\mu'=\mu+a \gamma$ for some $a\in \mathbb Z_{\geq 0}$ 
and $H^0(u, \alpha)_{\mu'}\neq 0$. 
By induction on $l(w)$,  $\mu'= \alpha $ and  the evaluation map 
$ev: H^0(u, \alpha) \longrightarrow \mathbb C_{\alpha} $ is surjective.
By (1), we see that $ev: H^0(u, \alpha)_{\alpha} \longrightarrow \mathbb C_{\alpha} $ is an isomorphism.
Since $\mu\in   M_{\geq 0}\setminus \{0\}$ and $\mu'=\alpha$, we have $a=0$ and hence 
$\mu= \alpha $. By the above arguments, the restriction of the evaluation map 
$ev:H^0(w, \alpha)_{\alpha} \longrightarrow H^0(u, \alpha)_{\alpha}$ is surjective.
Hence,  the evaluation map $ev:H^0(w, \alpha) \longrightarrow \mathbb C_{\alpha} $ is surjective.

  The other implication is straight forward.
\end{proof}

\begin{corollary}\label{cor1}

Let $w\in W$ and fix a reduced expression $w=s_{{i_1}}s_{{i_2}}\cdots s_{{i_r}}$. 
Let $\mu \in M_{\geq 0}\setminus \{0\}$. Then, we have 
\begin{enumerate} 
 \item
  $H^0(w, \alpha_{i_r})_{\mu} \neq 0$ if and only if $\mu=\alpha_{i_r}$ and 
  $\langle \alpha_{i_r}, \alpha_{i_k} 
\rangle = 0$ for $k=1,2,\cdots, r-1.$ 
\item In such a case, the evaluation map $ev: H^0(w, \alpha_{i_r})\longrightarrow
sl_{2, \alpha_{i_r}}$ is an isomorphism.
\end{enumerate}
\end{corollary}
\begin{proof} Proof of (1): Assume that $H^0(w, \alpha_{i_r})_{\mu} \neq 0$.
By Lemma \ref{lemma1}, we have $\mu =\alpha_{i_r}$ and the evaluation map  
$ev : H^0(w, \alpha_{i_r}) \longrightarrow \mathbb C_{\alpha_{i_r}}$ is surjective.
We now prove that  $\langle \alpha_{i_r}, \alpha_{i_k} \rangle = 0$ for $k=1,2,\cdots, r-1.$ 
Let $u=s_{{i_2}}s_{{i_3}}\cdots s_{{i_{r}}}$. Then, we have  $l(u)=l(w)-1$.
Since the evaluation map $ev: H^0(w, \alpha_{i_r})=H^0(s_{{i_1}}, H^0(u, \alpha_{i_r}))
\longrightarrow \mathbb C_{\alpha_{i_r}}$ is non zero, 
the evaluation map $ev: H^0(u, \alpha_{i_r})\longrightarrow \mathbb C_{\alpha_{i_r}} $ is non 
zero, because this evaluation map is the composition 
of the evaluation maps $H^0(s_{i_1}, H^0(u, \alpha))\longrightarrow H^0(u, \alpha)$ and
$H^0(u, \alpha)\longrightarrow \mathbb C_{\alpha}$.
By induction on $l(w)$, $\langle \alpha_{i_r}, \alpha_{i_k} \rangle = 0$ for $k=2,\cdots, r-1.$
Hence, $w=s_{{i_1}}s_{{i_r}}s_{{i_2}}\cdots s_{{i_{r-1}}}$ is also a reduced expression for $w$. 
In particular, $\alpha_{i_1}\neq \alpha_{i_{r}} $ and hence $\langle \alpha_{i_1}, \alpha_{i_r} 
\rangle \leq 0$. 
By Corollary \ref{commuting}, we have $H^0(w, \alpha_{i_r})=H^0(s_{{i_1}}s_{i_r}, \alpha_{i_r})$. 
Note that if $\langle \alpha_{i_1}, \alpha_{i_r} \rangle \leq -1$, by Lemma \ref{lemma2.3} we 
have $H^0(w, \alpha_{i_r})_{\alpha_{i_r}}= 0$, which is a contradiction.
Thus, we have $\langle \alpha_{i_1}, \alpha_{i_r} \rangle = 0$.
Hence $\langle \alpha_{i_r}, \alpha_{i_k} \rangle = 0$ for $k=1,2,\ldots, r-1$. 

The other implication follows from Corollary \ref{commuting}.

Assertion (2) follows from the fact that $H^0(s_{{i_r}}, \alpha_{i_r})$ is the $3$-dimensional 
cyclic $B$-submodule generated by a weight vector of weight $\alpha_{i_r}$.
\end{proof}


Let $\mathfrak p$ be a $B$-submodule of $\mathfrak g$ containing $\mathfrak b$.

  \begin{lemma} \label{lemma2}
  Let $w \in W$ and  let $\mu \in   M_{\geq 0}\setminus \{0\}$. 
  If $ H^0(w, \mathfrak g/ \mathfrak p)_{\mu} \neq 0$, then $\mu \in R^+$.
  \end{lemma}  
\begin{proof} If $l(w)=0$, there is nothing to prove.
Assume that $l(w)>0$. Then, we can choose $\gamma \in S$ such that $l(s_{\gamma}w)=l(w)-1$. 
Let $u=s_{\gamma}w$.
By {\it SES}, we have $H^0(w,\mathfrak g/ \mathfrak p)=H^0(s_{\gamma}, H^0(u, \mathfrak g/ 
\mathfrak p))$.

Since  $H^0(w, \mathfrak g/ \mathfrak p)_{\mu}\neq 0$, there exists an indecomposable 
$B_{\gamma}$-summand $V$ of 
$H^0(u, \mathfrak g/ \mathfrak p)$ such that $H^0(s_{\gamma}, V)_{\mu}\neq 0$.
Let $\mu'$ be the highest weight of $V$.  
By the same arguments as in the proof of Lemma \ref{lemma1}, we have $\mu=\mu'-a\gamma$ 
where $0\leq a \leq \langle \mu', \gamma \rangle$.

 Since $l(u)=l(w)-1$ and $V_{\mu'}\neq 0$, by induction on $l(w)$,  $\mu'\in R^{+}$.
Hence $\mu'-j\gamma\in R\cup \{0\}$ for 
every $0\leq j\leq \langle \mu', \gamma \rangle$ (see \cite[p.45, Section 9.4]{Hum1}).
Since $\mu \in  M_{\geq 0}\setminus \{0\}$, we have $\mu \in R^+$. 
\end{proof}

\begin{proposition} \label{cor2} Let $w\in W$ and fix a reduced expression 
$w=s_{{i_1}}s_{{i_2}}\cdots s_{{i_r}}$.
 Fix $1\leq j \leq r-1$. If $\langle \alpha_{i_j}, \alpha_{i_k} \rangle = 0$ 
 for every $1\leq k < j$,  
 then the natural map
$ H^0(w, \mathfrak g/ \mathfrak b)_{\alpha_{i_j}} \longrightarrow H^0(w, \mathfrak g/
\mathfrak p)_{\alpha_{i_j}}$ 
is surjective.
\end{proposition}
\begin{proof}
If $l(w)=0$, there is nothing to prove.
Assume $l(w)>0$ and let $u=s_{{i_1}}w$. Then, we have $l(u)=l(w)-1$.
By {\it SES}, we have the evaluation map $$ev:H^0(w,\mathfrak g/ \mathfrak p)=H^0(s_{{i_1}}, 
H^0(u, \mathfrak g/ \mathfrak p)) \longrightarrow H^0(u, \mathfrak g/ \mathfrak p).$$
We denote the restriction  of the evaluation map $ev$ to $H^0(w,\mathfrak g/ \mathfrak p)_
{\alpha_{i_j}}$ by $ev_1$.

First we will prove that $ev_1$ is an isomorphism.

Let $v$ be a non zero vector in $H^0(w, \mathfrak g/ \mathfrak p)$ of weight $\alpha_{i_j}$. 
Let $H^0(u, \mathfrak g/ \mathfrak p))\simeq \bigoplus_{i=1}^{m}V_i$ be a decomposition as a sum of indecomposable 
$B_{\alpha_{i_1}}$-submodules.
Since $v\in H^0(s_{{i_1}}, \bigoplus_{i=1}^{m}V_i)=\bigoplus_{i=1}^{m} H^0(s_{{i_1}}, V_i)$, $v=\sum_{i=1}^{m}v_i$ where 
$v_i\in H^0(s_{{i_1}}, V_i)$ ($1\leq i\leq m$), it follows that the weight of $v_i$ is same as the weight of $v$.
Hence, without loss of generality, we may assume that there exists an 
indecomposable $B_{\alpha_{i_1}}$-summand $V$ of $H^0(u, \mathfrak g/ \mathfrak p)$  such that 
$v\in H^0(s_{{i_1}}, V)_{\alpha_{i_j}}$.
Let $\mu$ be the highest weight of $V$. By the arguments as in the proof of Lemma \ref{lemma1}, 
$\mu=\alpha_{i_j}+a \alpha_{i_1}$ for some $a\in \mathbb Z_{\geq 0}$.
Since $H^0(u, \mathfrak g/ \mathfrak p)_{\mu}\neq0$, by Lemma \ref{lemma2} we see that $\mu$ is 
a positive root. 
Since either $j=1$, or $\langle \alpha_{i_j}, \alpha_{i_1} \rangle = 0$, 
we have $a=0$. Hence $V=\mathbb C.v$. Thus, the map 
$ev_1:H^0(w,\mathfrak g/ \mathfrak p)_{\alpha_{i_j}}\longrightarrow 
H^0(u, \mathfrak g/ \mathfrak p)_{\alpha_{i_j}}$ 
is injective. To prove $ev_1$  is surjective, 
let $v{'}$ be a non zero vector in $H^0(u, \mathfrak g/ \mathfrak p)$ of weight $\alpha_{i_j}$. 
By similar arguments, we may assume that there exists an
indecomposable $B_{\alpha_{i_1}}$-summand $V{'}$ of 
$H^0(u, \mathfrak g/ \mathfrak p)$ containing $v{'}$. Let $\mu'$ be the highest weight of $V'$.
Then, by the arguments as in the proof of Lemma \ref{lemma1}, 
$\mu'=\alpha_{i_j}+a \alpha_{i_1}$ for some $a\in \mathbb Z_{\geq 0}$. By the similar arguments as above, we 
see that 
$V^{'}=\mathbb C.v^{'}$. Hence, we conclude that $v^{'}$ 
is in the image of $ev_1$.

In particular, the restriction $ev_2: H^0(w,\mathfrak g/ \mathfrak b)_{\alpha_{i_j}}
\longrightarrow H^0(u, \mathfrak g/ \mathfrak b)_{\alpha_{i_j}}$
of the evaluation map $H^0(w,\mathfrak g/ \mathfrak b)\longrightarrow H^0(u, \mathfrak g/
\mathfrak b)$ is an isomorphism.

Now, consider the following commutative diagram of $T$-modules:
\begin{center}

$\xymatrix{ 
H^0(w, \mathfrak g/ \mathfrak b)_{\alpha_{i_j}}\ar[d]^{\reflectbox{\rotatebox[origin=c]{90}{$\sim$}}}_{ev_2} \ar[r]^{f} &
H^0(w, \mathfrak g/ \mathfrak p)_{\alpha_{i_j}} \ar[d]^{\reflectbox{\rotatebox[origin=c]{90}{$\sim$}}}_{ev_1} 
\\ H^0(u, \mathfrak g/\mathfrak b)_{\alpha_{i_j}} \ar[r]^{g}  & H^0(u, \mathfrak g/ \mathfrak p )
_{\alpha_{i_j}} }
$

\end{center}

By the induction on $l(w)$,   $ g: H^0(u, \mathfrak g/ \mathfrak b)_{\alpha_{i_j}} 
\longrightarrow H^0(u, 
\mathfrak g/ \mathfrak p )_{\alpha_{i_j}}$ is surjective.
By the commutativity of the above diagram, it follows that  
  the natural map $$f: H^0(w, \mathfrak g/ \mathfrak b)_{\alpha_{i_j}}\longrightarrow 
  H^0(w, \mathfrak g/\mathfrak p)_{\alpha_{i_j}}$$ is surjective.  This completes the proof.
  \end{proof}

\begin{corollary}\label{lemma3} Let $w \in W$ and fix a reduced expression  $w=s_{{i_1}}s_{{i_2}}
\cdots s_{{i_r}}$.
 Fix an integer $j\in \{1,\ldots, r-1\}$ such that for all $1\leq k < j$, $\langle \alpha_{i_j}, \alpha_{i_k} \rangle = 0$. 
  Then,
 $H^1(w, \alpha_{i_r})_{\alpha_{i_j}}=0$.
\end{corollary}
 \begin{proof}  Let $\alpha=\alpha_{i_r}$. Now look at following short exact sequence of 
 $B$-modules: 
$$0\longrightarrow \mathfrak g_{\alpha} \longrightarrow \mathfrak g / \mathfrak b \longrightarrow 
\mathfrak g / \mathfrak p_{\alpha} \longrightarrow 0$$
Note that by  Theorem \ref{AS}, $H^1(w, \mathfrak g/ \mathfrak b)=0$.
Applying $H^0(w, - )$ to the above short exact sequence of $B$-modules and taking the 
$\alpha_{i_j}$ weight spaces, we have  the exact sequence of $T$-modules: 
$$0\longrightarrow H^0(w, \alpha)_{\alpha_{i_j}} \longrightarrow H^0(w, \mathfrak g/ \mathfrak b)_
{\alpha_{i_j}} \longrightarrow H^0(w, \mathfrak g/ \mathfrak p_{\alpha})_{\alpha_{i_j}}
\longrightarrow H^1(w, \alpha)_{\alpha_{i_j}} \longrightarrow 0$$
By Proposition \ref{cor2},  we conclude that $H^1(w, \alpha)_{\alpha_{i_j}}=0$. This completes 
the proof.
 \end{proof}

 \section{Action of the minimal Parabolic subgroup $P_{\alpha_{i_1}}$ on $Z(w, \underline i)$}

    Recall that $\phi_w$ denotes the birational morphism $Z(w, \underline i) \longrightarrow X(w)$. As
    in Section 2, the composition of inclusion $X(w)$ in $G/B$
with $\phi_w$ will also be denoted by $\phi_w$. Further, we denote the tangent bundle of 
$Z(w,\underline i)$ 
by $T_{(w,\underline i)}$, where 
 $\underline i=(i_1, i_2, \ldots, i_r)$.
By using the differential map, we see that $T_{(w, \underline i)}$ is a subsheaf  of 
$\phi_w^*(T_{G/B})$.
Hence $H^0(Z(w, \underline i), T_{(w, \underline i)})$ is a $B$-submodule of  $H^0(Z(w, 
\underline i), \phi_w^*(T_{G/B}))$.
 
Since the tangent bundle of $G/B$ is the homogeneous vector bundle associated to the 
representation $\mathfrak g/ \mathfrak b$ of $B$, we have   $$H^0(Z(w, \underline i), 
\phi_w^*(T_{G/B}))= H^0(w,\mathfrak g/\mathfrak b ).$$
Therefore, $H^0(Z(w, \underline i), T_{(w, \underline i)})$ is a $B$-submodule of
$H^0(w,\mathfrak g/\mathfrak b )$. 

Denote by $\mathfrak p_{\alpha_{i_1}}$, the Lie algebra of the minimal parabolic subgroup 
$P_{\alpha_{i_1}}$ of $G$ containing $B$. Note that $\mathfrak b$ is contained in $\mathfrak p_{\alpha_{i_1}}$.
  
  \begin{lemma}\label{l3}
 Let $w=s_{{i_1}}\cdots s_{{i_r}}$ be a reduced expression $\underline i$ for $w$. Then, 
 \begin{enumerate} 
 \item
  There is a non zero homomorphism $f_w: \mathfrak p_{\alpha_{i_1}}\longrightarrow
 H^0(Z(w, \underline i), T_{(w, \underline i)})$ of $B$-modules (which is also a homomorphism of 
 Lie algebras).
 \item If $w=w_0$, the homomorphism $f_{w_0}:\mathfrak p_{\alpha_{i_1}}\longrightarrow H^0(Z(w_0, 
 \underline i), T_{(w_0, \underline i)})$ in (1) is injective.  
\end{enumerate}
 
 \end{lemma}
\begin{proof}Proof of (1):
 Consider the action of $P_{\alpha_{i_1}}$ on $Z(w, \underline i)$ induced by the following 
 left action of $P_{\alpha_{i_1}}$
 on $P_{\alpha_{i_1}}\times P_{\alpha_{i_2}}\times \cdots \times P_{\alpha_{i_r}}$:\\
 Let $p\in P_{\alpha_{i_1}}$ and $x=(p_1, p_2, \cdots, p_r) \in P_{\alpha_{i_1}}\times
 P_{\alpha_{i_2}}\times \cdots\times  P_{\alpha_{i_r}} $
then $p.x:=(pp_1, p_2, \cdots, p_r)$.

Clearly, this action is non trivial.  
Hence, there is a non trivial homomorphism $$\psi_w:P_{\alpha_{i_1}} 
\longrightarrow Aut^0(Z(w, \underline i))$$ of algebraic groups. Consider the action of $B$ on $P_{\alpha_{i_1}}$
by conjugation and the action of $B$ on $Aut^0(Z(w, \underline i))$ via $\psi_w$. Note that $\psi_w$ is $B$-equivariant.

By  \cite[Theorem 3.7]{Mat}, $Aut^0(Z(w, \underline i))$ is an algebraic group 
and $$Lie(Aut^0(Z(w, \underline i)))=H^0(Z(w, \underline i), T_{(w, \underline i)}).$$
Then, the induced homomorphism $$f_w:\mathfrak p_{\alpha_{i_1}}\longrightarrow H^0(Z(w, \underline i)
, T_{(w, \underline i)})$$ of $B$-modules (homomorphism of Lie algebras) is non zero. 

Proof of (2):
Since $f_w:\mathfrak p_{\alpha_{i_1}}\longrightarrow H^0(Z(w, \underline i), T_{(w, \underline i)})$
is a
non zero homomorphism of $B$-modules (homomorphism of Lie algebras),  
$f_w(\mathfrak p_{\alpha_{i_1}})$ contains a $B$-stable line $L$. Let $\mu$
be the character of $B$ such that $b.v=\mu(b).v$ for all $b\in B$ and for all $v\in L$.
That is, $L$ is the one-dimensional space generated by a lowest weight vector of weight $\mu$.

Since $w=w_0$,  $H^0(Z(w_0, \underline i), T_{(w_0, \underline i)})$ is a $B$-submodule 
of $H^0(G/B, T_{G/B})$. By Bott's theorem \cite[Theorem VII]{Bott1} we have 
$H^0(G/B, T_{G/B})=\mathfrak g$.
Hence $H^0(Z(w_0, \underline i), T_{(w_0, \underline i)})$ is a $B$-submodule of $\mathfrak g$.
Since there is a unique $B$-stable one-dimensional subspace $L$ of $\mathfrak g$ 
and the character of $B$ is $-\alpha_0$, we conclude that $\mu=-\alpha_0$ and $L=\mathfrak g_{-\alpha_0} \subset f_{w_0}(\mathfrak p_{\alpha_{i_1}})$.
By the similar arguments, the unique $B$-stable one-dimensional subspace in $\mathfrak p_{\alpha_{i_1}}$ 
is $\mathfrak g_{-\alpha_0}$.
  Hence $f_{w_0}$ is injective (otherwise $Ker (f_{w_0})\neq 0$ and hence the unique $B$-stable line 
  $\mathfrak g_{-\alpha_0}$ is a subspace of $Ker (f_{w_0})$, which is a contradiction).
 \end{proof}

   \begin{corollary}\label{C1}{\  }
  \begin{enumerate}  
    \item 
  $H^0(Z(w_0, \underline i), T_{(w_0, \underline i)})$ is a Lie subalgebra of $\mathfrak g$.
  \item
  Any Borel subalgebra of $H^0(Z(w_0, \underline i), T_{(w_0, \underline i)})$ is isomorphic to 
  $\mathfrak b$.
 \item Any maximal toral subalgebra of $H^0(Z(w_0, \underline i), T_{(w_0, \underline i)})$ is 
 isomorphic to $\mathfrak h$.
 
 \end{enumerate}
\end{corollary}
\begin{proof}

Proof of (1):
 By Lemma \ref{l3}(2), $\mathfrak b$ is a Lie subalgebra of $H^0(Z(w_0, \underline i),
 T_{(w_0, \underline i)})$. 
  Since $H^0(Z(w_0, \underline i), T_{(w_0, \underline i)})$ is a $B$-submodule of 
$\mathfrak g$, for any $Y\in H^0(Z(w_0, \underline i), T_{(w_0, \underline i)})$ and for any $X \in \mathfrak b$ the Lie bracket 
$[X, Y]$ in  $\mathfrak g$ is same as the Lie bracket 
in $H^0(Z(w_0, \underline i), T_{(w_0, \underline i)})$.
It remains to prove that for every $\alpha, \beta \in R^+$ such that $\alpha, \beta$ are weights of 
$H^0(Z(w_0, \underline i), T_{(w_0, \underline i)})$, the Lie bracket $[x_{\beta}, x_{\alpha}]$ in  $\mathfrak g$ is same as the Lie bracket 
in $H^0(Z(w_0, \underline i), T_{(w_0, \underline i)})$.

Note that the Lie subalgebra of $H^0(Z(w_0, \underline i), T_{(w_0, \underline i)})$ generated by 
$\mathfrak g_{\beta}\cap H^0(Z(w_0, \underline i), T_{(w_0, \underline i)})$ for $\beta \in R^+$ is same as the  
Lie subalgebra generated by $\mathfrak g_{\alpha}\cap H^0(Z(w_0, \underline i), T_{(w_0, \underline i)})$ for $\alpha \in S$. 
Hence it is enough to prove that for every $\beta \in R^+$ and $\alpha\in S$ such that $\alpha, \beta$ are weights of 
$H^0(Z(w_0, \underline i), T_{(w_0, \underline i)})$, the Lie bracket $[x_{\beta}, x_{\alpha}]$ in  $\mathfrak g$ is same as the Lie bracket 
in $H^0(Z(w_0, \underline i), T_{(w_0, \underline i)})$.

Let $[-,-]'$ be the Lie bracket in $H^0(Z(w_0, \underline i), 
T_{(w_0, \underline i)})$.
For $\beta \in R^+, \alpha \in S$, by Jacobi identity
we have $$[x_{-\beta}, [x_{\beta}, x_{\alpha}]']'=[[x_{-\beta}, x_{\beta}]', x_{\alpha}]'+[x_{\beta}, [x_{-\beta}, x_{\alpha}]']'.$$
Since $x_{-\beta}\in\mathfrak b$ and $\mathfrak b$ is a Lie subalgebra of $H^0(Z(w_0, \underline i), T_{(w_0, \underline i)})$
and $H^0(Z(w_0, \underline i), T_{(w_0, \underline i)})$ is a $B$-submodule of $\mathfrak g$, we have 
$$[x_{-\beta}, x_{\beta}]'=[x_{-\beta}, x_{\beta}] ~ and~ 
[x_{-\beta}, x_{\alpha}]'=[x_{-\beta}, x_{\alpha}].$$
Hence, we have 
\begin{equation}
 [x_{-\beta}, [x_{\beta}, x_{\alpha}]']'=[[x_{-\beta}, x_{\beta}], x_{\alpha}]'+[x_{\beta}, [x_{-\beta}, x_{\alpha}]]'.
\end{equation}

Note that $[x_{-\beta}, x_{\beta}], [x_{-\beta}, x_{\alpha}]\in \mathfrak b$.
Therefore, by (5.1) and Jacobi identity we have  
$$[x_{-\beta}, [x_{\beta}, x_{\alpha}]']'=[[x_{-\beta}, x_{\beta}], x_{\alpha}]+[x_{\beta}, [x_{-\beta}, x_{\alpha}]]=
[x_{-\beta}, [x_{\beta}, x_{\alpha}]].$$

Since $x_{-\beta}\in \mathfrak b$, we have $[x_{-\beta}, [x_{\beta}, x_{\alpha}]']'=[x_{-\beta}, [x_{\beta}, x_{\alpha}]']$.
Hence, we have  

\begin{equation}
 [x_{-\beta}, [x_{\beta}, x_{\alpha}]']=[x_{-\beta}, [x_{\beta}, x_{\alpha}]].
\end{equation}
If $[x_{\beta}, x_{\alpha}]=0$, then $\alpha+\beta \notin R$. In particular,
$H^0(Z(w_0, \underline i), T_{(w_0, \underline i)})_{\alpha+\beta}=0$.
Then, we have  $[x_{\beta}, x_{\alpha}]'=0$.

If $[x_{\beta}, x_{\alpha}]\neq 0$, then $\alpha+\beta \in R$.
$$[x_{-\beta}, [x_{\beta}, x_{\alpha}]]=[x_{\beta}, [x_{-\beta}, x_{\alpha}]]-h_{\beta}\cdot x_{\alpha}.$$
If $[x_{-\beta}, x_{\alpha}]=0$ and $h_{\beta}\cdot x_{\alpha}=0$, then $\alpha, \beta$ are orthogonal
and $\beta-\alpha\notin R$. Hence, we have $\alpha+\beta \notin R$. This contradicts the 
assumption. 
Hence, we have $[x_{\beta}, x_{\alpha}]'=c_1x_{\alpha+\beta}$ and $[x_{\beta}, x_{\alpha}]=c_2x_{\alpha+\beta}$, with $c_2\neq 0$.
Therefore, by (5.2) it follows that $c_1=c_2$ and $[x_{\beta}, x_{\alpha}]'=[x_{\beta}, x_{\alpha}]$.
Hence
$H^0(Z(w_0, \underline i), T_{(w_0, \underline i)})$ is a Lie subalgebra of $\mathfrak g$.

 Proof of (2): By (1), $H^0(Z(w_0, \underline i), T_{(w_0, \underline i)})$ is a Lie subalgebra of $\mathfrak g$.
Now we claim that $\mathfrak b$ is a Borel subalgebra of $H^0(Z(w_0, \underline i), 
T_{(w_0, \underline i)})$.
 Otherwise, there exists a Borel subalgebra $\mathfrak b'$ of $H^0(Z(w_0, \underline i), 
 T_{(w_0, \underline i)})$ properly containing 
$\mathfrak b$. Since $H^0(Z(w_0, \underline i), T_{(w_0, \underline i)})$ is a Lie subalgebra of 
$\mathfrak g$, we see that $\mathfrak g_{\alpha}\subset \mathfrak b'$ for some 
simple root $\alpha$.
 Since $\mathfrak b$ is a Borel subalgebra of $\mathfrak g$
 and $\mathfrak b$ is a Lie subalgebra of $H^0(Z(w_0, \underline i), T_{(w_0, \underline i)})$, the simple Lie algebra $sl_{2, \alpha}$ is a Lie subalgebra of $\mathfrak b'$,
which is a contradiction to the solvability of $\mathfrak b'$. Hence $\mathfrak b$ is a 
Borel subalgebra of $H^0(Z(w_0, \underline i), T_{(w_0, \underline i)})$.
Since any two  Borel subalgebras of $H^0(Z(w_0, \underline i), T_{(w_0, \underline i)})$ are
conjugate (see \cite[p.84, Theorem 16.4]{Hum1}),  we conclude (2).

Proof of (3): Since any two maximal toral subalgebras of $H^0(Z(w_0, \underline i), 
T_{(w_0, \underline i)})$ are conjugate (see \cite[p.84, Corollary 16.4]{Hum1}), the
proof follows from (2).
\end{proof}

Let $w\in W$, let $w=s_{{i_1}}s_{{i_2}}\cdots s_{{i_r}}$
be a reduced expression $\underline i$ of $w$.
 Fix a reduced expression $w_0=s_{{j_1}}s_{{j_2}}\cdots s_{{j_r}}s_{j_{r+1}}\cdots s_{j_N}$ of 
$w_0$  such that $\underline j=(j_1,j_2,\ldots j_N)$ and $\underline i=(j_1,j_2,\ldots j_r)$.
Let $v=s_{{j_{r+1}}}s_{{j_{r+2}}}\cdots
s_{{j_N}}$ and  $\underline j'=(j_{r+1},\ldots, j_N)$.  

 Since the  $Z(v, \underline j')$-fibration $Z(w_{0}, \underline j)\longrightarrow Z(w, \underline i)$
is $P_{\alpha_{i_{1}}}$ equivariant, it follows that $$H^{0}(Z(w_{0} ,
\underline j), T_{(w_{0}, \underline j)})
\longrightarrow H^{0}(Z(w , \underline i) , T_{(w, \underline i)})$$ is a
homomorphism of $P_{\alpha_{i_{1}}}$-modules. Hence, it is a homomorphism
of $\mathfrak p_{\alpha_{i_{1}}}$-modules.
Thus, the restriction of this map to $\mathfrak p_{\alpha_{i_{1}}}$ is the same as
the map induced by the action of $P_{\alpha_{i_{1}}}$ on $Z(w, \underline i).$

Note that since  $f_{w_0}:\mathfrak p_{\alpha_{i_1}}\longrightarrow H^0(Z(w_0, 
 \underline i), T_{(w_0, \underline i)})$ is injective (see Lemma \ref{l3}(2)),  we identify  $\mathfrak p_{\alpha_{i_1}}$ as a Lie subalgebra of $H^0(Z(w_0, 
 \underline i), T_{(w_0, \underline i)})$.

Hence, we have the following commutative diagram of $P_{\alpha_{i_{1}}}$-modules:

 \begin{center}  
  $\xymatrix{\mathfrak p_{\alpha_{i_1}} \ar[dr]_{f_w} \ar@{^{(}->}[r] & H^{0}(Z(w_0, \underline j) , T_{(w_0 , \underline j)}) \ar[d]\\
&H^0(Z(w, \underline i), T_{(w, \underline i)}) 
} $
\end{center}
Further, the maps in the above diagram are homomorphisms of Lie algebras.

  For simplicity of notation, we denote both the natural map $$H^{0}(Z(w_0, \underline j) , T_{(w_0 , \underline j)}) \longrightarrow H^0(Z(w, \underline i), T_{(w, \underline i)})$$
 and its restriction to $\mathfrak p_{\alpha_{i_1}}$ by $f_w$.

 Let $d(w)$ be the  number of distinct $i_j$'s in $ \underline i=(i_1,i_2,\ldots,i_r)$ 
(i.e, the number of 
distinct simple reflections $s_{{i_j}}$'s appearing in the reduced expression $\underline i$ of 
$w$).
Let $\leq$ be the Bruhat-Chevalley ordering on $W$.
Note that $d(w)$ is equal to the number of distinct Schubert curves in $X(w)$.
That is, $d(w)$ is equal to the number of distinct $j\in \{1,2,\ldots, n\}$ such that $s_j\leq w$.  
In particular, it 
is independent of the choice of the reduced expression $\underline i$ of $w$.
Further, we also note that $d(w_0)=n$.

 Now, we prove the following Lemma:
\begin{lemma}\label{torus} {\  }

\begin{enumerate}
 \item 
The dimension of the zero weight space 
$H^0(Z(w, \underline i), T_{(w, \underline i)})_0$ is at most $d(w)$.
\item 
In particular, $dim(H^0(Z(w, \underline i), T_{(w, \underline i)})_0)\leq rank(G)$. 

\end{enumerate}

 \end{lemma}
\begin{proof}
Consider the following short exact sequence  of $B$-modules:
$$0\longrightarrow \mathfrak b \longrightarrow \mathfrak g \longrightarrow \mathfrak g/ \mathfrak
b \longrightarrow 0 
$$
 By applying $H^0(w, -)$  to the above short exact sequence, we have the following exact sequence
 of $B$-modules:
 $$0\longrightarrow H^0(w, \mathfrak b)\longrightarrow H^0(w, \mathfrak g)\longrightarrow 
 H^0(w, \mathfrak g/\mathfrak b)\longrightarrow H^1(w, \mathfrak b)\longrightarrow 0$$  
(Note that $H^1(w, \mathfrak g)=0$ (see \cite[Lemma $2.5(2)$]{Ka4}))

 By Lemma \ref{l2}, we have $H^1(w, \mathfrak b)_0=0$. Since $H^0(w, \mathfrak g)=\mathfrak g$,
by taking the zero weight space to the above exact sequence  
we have the following short exact sequence of $T$-modules;
$$ 0\longrightarrow H^0(w, \mathfrak b)_0\longrightarrow \mathfrak h \overset{\phi}
\longrightarrow H^0(w, \mathfrak g/\mathfrak b)_0\longrightarrow 0$$

Claim: $dim (H^0(w, \mathfrak b)_0)=rank(G)-d(w).$

We use the similar arguments as in the proof of Lemma \ref{l2} and Lemma \ref{1} to prove the claim.

Let
$w=s_{{i_1}}s_{{i_2}}\cdots s_{{i_r}}$
be a reduced expression $\underline i$ of $w$.
Since $S$ is a basis for the complex vector space $\mathfrak h$, for every 
$1\leq j \leq n$ there exists a $h(\alpha_{j})\in \mathfrak h$
such that $\alpha_i(h(\alpha_j))=\delta_{i,j}$ for $1\leq i\leq n$.
First note that for every $i\neq j$, the one-dimensional subspace 
$\mathbb C h(\alpha_j)$ of $\mathfrak h$
is an indecomposable $B_{\alpha_{i}}$-direct summand of $\mathfrak b$.
Therefore, the image of the evaluation map 
$ev:H^0(w,\mathfrak b)\longrightarrow \mathfrak b$ contains $h(\alpha_{j})$
for every $1\leq j\leq n$ such that $s_j\nleq w$.
Let $1\leq k\leq n$ such that $s_k\leq w$. Let $1\leq j_0\leq r$
be  the largest integer such that $i_{j_0}=k$, let $u=s_{i_{j_0+1}}\cdots s_{i_r}$.
Note that since $\alpha_{i_j}(h(\alpha_k))=0$ for $j_0+1\leq j\leq r$, $\mathbb C h_{\alpha_k}$ is 
contained in the image of the evaluation map $ev:H^0(u, \mathfrak b)\longrightarrow \mathfrak b$.
Therefore, $\mathbb C h(\alpha_k)\oplus \mathbb C_{-\alpha_k}$ is an indecomposable $B_{\alpha_k}
$-direct summand of $H^0(u, \mathfrak b)$ (see \cite[Lemma 3.3]{Ka4}).

Further, by Lemma \ref{lemma2.4} $$\mathbb C h(\alpha_k)\oplus \mathbb C_{-\alpha_k}=
V\otimes \mathbb C_{-\omega_k}$$ where $V$ is the standard $2$-
dimensional representation of $\widetilde{L}_{\alpha_k}.$
Therefore, by Lemma \ref{lemma2.3} and Lemma \ref{lemma2.2}, 
$H^0(s_{i_{j_0}}, \mathbb C h(\alpha_k)\oplus \mathbb C_{-\alpha_k} )=0$.

Let $v=s_{i_{j_0}}u$.
By {\it SES}, we conclude that  $H^0(v, \mathfrak b)_0 = \bigoplus_{\{i: s_i\nleq v\}}\mathbb C h(\alpha_i)$.
In view of \cite[Lemma 2.6]{Ka4}, 
$H^0(w, \mathfrak b)_0 =\bigoplus_{\{i: s_i\nleq w\}}\mathbb C h(\alpha_i)$.

  Then by the above claim and the short exact sequence, we have $$dim (H^0(w, \mathfrak g/ \mathfrak b)_0)=d(w).$$
  Since $H^0(Z(w, \underline i), T_{(w, \underline i)})$ is a $B$-submodule of 
 $H^0(w, \mathfrak g/ \mathfrak b)$, we have $$dim(H^0(Z(w, \underline i), T_{(w, \underline i)})_0)
 \leq d(w).$$
\end{proof}

\section {The $B$-module of the global sections of the tangent bundle on $Z(w, \underline i)$} 
In this section,
we study the $B$-module of the global sections of the tangent bundle on $Z(w, \underline i)$.
In particular, we prove that the dimension of the zero weight space of $H^0(Z(w, \underline i),
T_{(w, \underline i)})$
is equal to $d(w)$, the number  of Schubert curves in $X(w)$.
We also prove that  $H^0(Z(w, \underline i), T_{(w, \underline i)})$ contains a Lie subalgebra $\mathfrak b'$
  isomorphic to $\mathfrak b$ if and only if $w^{-1}(\alpha_{0})<0$.

We use the notation as in the previous section.

Let $w \in W$ and fix a reduced expression 
$w=s_{{i_1}}s_{{i_2}}\cdots s_{{i_r}}$.
 Let $supp(w):=\{j\in \{1,2,\ldots, n\} : s_j\leq w\}$, the support of $w$.
Note that $d(w)=|supp(w)|$.
 
 We have the following proposition:

\begin{proposition} \label{Prop0} {\ }
 \begin{enumerate}
  \item  
 $\{f_w(h_{\alpha_{i_j}}): j\in supp(w) \}$ forms a basis of $H^0(Z(w, \underline i), T_{(w,\underline i)})_{0}$.
\item In particular, $dim(H^0(Z(w, \underline i), T_{(w,\underline i)})_{0})=d(w)$.
 \item 
 The image $f_w(\mathfrak h)$ is a maximal toral subalgebra of $H^0(Z(w, \underline i), 
 T_{(w,\underline i)})$.


 \end{enumerate}

\end{proposition}
\begin{proof}
  If $w=w_0$,  then by Lemma \ref{l3}(2), $f_{w_0}$ is injective and by Corollary \ref{C1}, $$dim(H^0(Z(w_0, \underline i), T_{(w_0, \underline i)}))_0
  )= rank(G)=d(w_0).$$ Hence, $\{f_{w_0}(h_{\alpha_{i_j}}):j\in supp(w_0)\}$ forms a basis of $H^0(Z(w_0, \underline i), T_{(w_0,\underline i)})_{0}$.

  Otherwise, choose a reduced expression $w_0=s_{{j_1}}s_{{j_2}}\ldots s_{{j_N}}$ of
  $w_0$ such that $(j_1, j_2, \cdots, j_r)=\underline i$.
Let $v=s_{{j_1}}s_{{j_2}}\cdots s_{{j_{r+1}}}$ and $\underline i'=(j_1, \ldots j_r, j_{r+1})$.
Note that $l(v)=l(w)+1$.
 By descending induction on $l(w)$,  $\{f_v(h_{\alpha_{i_j}}):j\in supp(v)\}$ forms a basis of $H^0(Z(v, \underline i'), T_{(v,\underline i')})_{0}$
 and  $$dim(H^0(Z(v, \underline i'), T_{(v, \underline i')}))_0)
 = d(v).$$

 Note that by Lemma \ref{torus}, $dim (H^0(Z(w, \underline i), T_{(w, \underline i)})_0) \leq d(w)$.
  By using $\it{LES}$ and Lemma \ref{vanishing}, we have the following exact sequence of 
$B$-modules:
 
 $0\longrightarrow H^0(v, \alpha_{i_{r+1}})\longrightarrow H^0(Z(v,\underline i'), 
 T_{(v,\underline i')})\longrightarrow H^0(Z(w,\underline i), T_{(w, \underline i)})
 \longrightarrow H^1(v, \alpha_{i_{r+1}}) 
 \longrightarrow \\ H^1(Z(v,\underline i'), T_{(v, \underline i')})\longrightarrow 
 H^1(Z(w, \underline i), T_{(w,\underline i)})
 \longrightarrow 0.$

 By taking the zero weight spaces, we have the following exact sequence of $T$-modules:
  
  $$0\longrightarrow H^0(v, \alpha_{i_{r+1}})_0 \longrightarrow H^0(Z(v, \underline i'), 
  T_{(v, \underline i')})_0\longrightarrow H^0(Z(w, \underline i), T_{(w, \underline i)})_0 
  \longrightarrow H^1(v, \alpha_{i_{r+1}})_0 \cdots $$

  First assume that there exists $1\leq j \leq r$ such that $\alpha_{i_j}=\alpha_{i_{r+1}}$, 
  so that $d(v)=d(w)$. 
 By Lemma \ref{1}, we have $H^0(v, \alpha_{i_{r+1}})_0=0$. 
   Hence $$d(v)= dim(H^0(Z(v, \underline i'), T_{(v, \underline i')}))_0) \leq 
   dim (H^0(Z(w, \underline i), T_{(w, \underline i)})_0 )\leq d(w).$$
  
  Since $d(w)=d(v)$,  we have 
   $$dim (H^0(Z(v, \underline i'), T_{(v, \underline i')})_0) =d(v)=d(w) =d
   im (H^0(Z(w, \underline i), T_{(w, \underline i)})_0).$$
  Hence, by the above exact sequence, we conclude that 
   $\{f_w(h_{\alpha_{i_j}}): j\in supp(w)\}$ forms a basis of $H^0(Z(w, \underline i), T_{(w,\underline i)})_{0}$.
  
 Otherwise $d(w)=d(v)-1$ and by Lemma \ref{1}(2), we see that $H^0(v, \alpha_{i_{r+1}})_0=\mathbb C.h_{\alpha_{i_{r+1}}}$.
 By using the above exact sequence, 
 we see that $$dim (H^0(Z(w, \underline i), T_{(w, \underline i)})_0)\geq d(v)-1.$$
 Since $dim (H^0(Z(w, \underline i), T_{(w, \underline i)})_0) \leq d(w)$, we conclude that
 $$dim (H^0(Z(w, \underline i), T_{(w, \underline i)})_0)=d(w)$$ and hence
    $\{f_w(h_{\alpha_{i_j}}): j\in supp(w)\}$ forms a basis of $H^0(Z(w, \underline i), T_{(w,\underline i)})_{0}$.
 This completes the proof of (1) and (2).
 
Proof of (3):

By Lemma \ref{l3}(2),
$f_{w_0}:\mathfrak p_{\alpha_{i_1}}\longrightarrow H^0(Z(w_0, 
 \underline j), T_{(w_0, \underline j)})$ is an injective homomorphism of Lie algebras.
 By Corollary \ref{C1}(1),  $H^0(Z(w_0, 
 \underline j), T_{(w_0, \underline j)})$ is a Lie subalgebra of $\mathfrak g$.
 Hence, we have $$H^0(Z(w_0, \underline j), T_{(w_0, \underline j)})_0=\mathfrak h.$$

Let $u=s_{{i_1}}s_{{i_2}}\cdots s_{{i_{r-1}}}$ and $\underline i'=(i_1,i_2,\ldots, i_{r-1})$.
Note that $l(u)=l(w)-1$. 

Consider the homomorphism $f:H^0(Z(w, \underline i), T_{(w, \underline i)})\longrightarrow H^0(Z(u, \underline i'), T_{(u, \underline i')})$ of Lie algebras
induced by the $\mathbb P^1$-fibration $f_r:Z(w, \underline i) \longrightarrow Z(u,\underline i')$
as in Section $2$.
By {\it LES}, $Ker(f)= H^0(w, \alpha_{i_r})$.

Note that by Lemma \ref{l1}(1), 
\begin{equation}
 H^0(w, \alpha_{i_r})_{\mu}=0 ~ unless ~\mu\leq \alpha_{i_r}. 
 \end{equation}

Case 1: If $s_{i_r}\leq u$, then  by Lemma \ref{1}(1), $H^0(w, \alpha_{i_r})_0=0$. 
Hence by Corollary \ref{cor1} and Lemma \ref{l1}(2), we conclude that $H^0(w, \alpha_{i_r})_{\mu}=0$
unless $\mu\in R^-$.
Since for every $\beta \in R^+$, $ad(x_{-\beta})^r=0$ in $H^0(Z(w, \underline i), T_{(w, \underline i)})$
for some $r\in \mathbb N$ (since for every positive root $\alpha$, there is a $r\in \mathbb N$ such that 
$\alpha+k\beta \notin R$ for all $k\geq r$), we conclude that 
every element of $H^0(w, \alpha_{i_r})\subseteq H^0(Z(w, \underline i), T_{(w, \underline i)})$ is nilpotent.

Case 2:
Assume that $s_{i_r}\nleq u$.

Sub case (a): If $\langle \alpha_{i_j}, \alpha_{i_r} \rangle \neq 0$ for some $1\leq j \leq r-1$, then
by Corollary \ref{cor1}(1), we have $H^0(w, \alpha_{i_r})_{\alpha_{i_r}}=0$.
Hence by (6.1),  we have $H^0(w, \alpha_{i_r})_{\mu}=0$ unless $\mu \leq 0$. Therefore, again
by Lemma \ref{l1}(2) $H^0(w, \alpha_{i_r})_{\mu}=0$ unless $\mu \in R^-\cup \{0\}$.
Further, by Lemma \ref{1}, $H^0(w, \alpha_{i_r})_0=\mathbb C.h_{\alpha_{i_r}}$.
Hence, a maximal toral subalgebra of $H^0(w, \alpha_{i_r})\subseteq H^0(Z(w, \underline i), T_{(w, \underline i)})$ lies in
$\mathbb C.h_{\alpha_{i_r}}\oplus \mathbb C_{-\alpha_{i_r}}$ and so it is one-dimensional.

Sub case (b): If $\langle \alpha_{i_j}, \alpha_{i_r} \rangle = 0$ for all $1\leq j \leq r-1$,
then by Corollary \ref{commuting}, we have  $$H^0(w, \alpha_{i_r})\simeq sl_{2, \alpha_{i_r}}.$$
Hence, any maximal toral subalgebra of the ideal $H^0(w, \alpha_{i_r}) \subseteq H^0(Z(w, \underline i), T_{(w, \underline i)})$ 
lies in $sl_{2, \alpha_{i_r}}$ and so  it is one-dimensional.

Hence, it follows that $$f_w(\mathfrak h)\cap Ker(f)=Ker(f)_0=H^0(w, \alpha_{i_r})_0$$ is a maximal toral
subalgebra of $Ker(f)$ and its dimension is at most one.

By induction on $l(w)$ and by (1), $f_u(\mathfrak h)=H^0(Z(u, \underline i'), T_{(u, \underline i')})_0$ is a maximal toral subalgebra 
of $H^0(Z(u, \underline i'), T_{(u, \underline i')})$.

 Now, consider the following commutative diagram of Lie algebras:
 \begin{center}  
  $\xymatrix{H^{0}(Z(w_0, \underline j) , T_{(w_0 , \underline j)}) \ar[d]_{f_w} \ar[dr]^{f_u} &\\
H^0(Z(w, \underline i), T_{(w, \underline i)})\ar[r]^{f}
  &H^0(Z(u, \underline i'), T_{(u, \underline i')})
} $
\end{center}  
 
 Note that by commutativity of the above diagram and by (1), it follows that $f_w(\mathfrak h)$ is 
 an extension of $f_u(\mathfrak h)$ and $f_w(\mathfrak h)\cap Ker(f)$.
Thus, we conclude that $f_w(\mathfrak h) = H^0(Z(w, \underline i), T_{(w, \underline i)})_0$
is a maximal toral subalgebra of $H^0(Z(w, \underline i), T_{(w, \underline i)})$.
 This completes the proof of the proposition.

 \end{proof}

 Consider the restriction of the homomorphism $f_w:\mathfrak p_{\alpha_{i_1}}\longrightarrow
 H^0(Z(w, \underline i), T_{(w, \underline i)})$(as in Lemma \ref{l3}) to $\mathfrak b$ and denote it also by $f_w$.

\begin{lemma}\label{Borel}
The homomorphism $f_w: \mathfrak b \longrightarrow H^0(Z(w,\underline i), T_{(w, \underline i)})$
  is injective if and only if $w^{-1}(\alpha_0)<0$.
\end{lemma}
\begin{proof}
Assume that  $f_w$ is injective.
Since  $ H^0(Z(w, \underline i), T_{(w, \underline i)})$ is a $B$-submodule of 
$H^0(w, \mathfrak g/ \mathfrak b)$,
 we have $ H^0(w, \mathfrak g/ \mathfrak b)_{-\alpha_0}\neq 0$.

  Recall from the proof the Lemma \ref{torus}, the following exact sequence of $B$-modules:
 $$0\longrightarrow H^0(w, \mathfrak b)\longrightarrow \mathfrak g \longrightarrow 
 H^0(w, \mathfrak g/\mathfrak b)\longrightarrow H^1(w, \mathfrak b)\longrightarrow 0$$

Note that if $G$ is simply laced, by \cite[Lemma 3.4]{Ka4} $H^1(w, \mathfrak b)=0$. 
If $G$ is non simply laced, since $-\alpha_0$ is a long root by  \cite[Lemma 4.8(2)]{Ka4},
we have $H^1(w, \mathfrak b)_{-\alpha_0}=0$. 
Hence, we have the following short exact sequence of $T$-modules:
 $$0\longrightarrow H^0(w, \mathfrak b)_{-\alpha_0}\longrightarrow \mathfrak
 g_{-\alpha_0}\longrightarrow H^0(w, \mathfrak g/\mathfrak b)_{-\alpha_0}\longrightarrow 0$$  

Since $dim(\mathfrak g_{-\alpha_0})=1$, $H^0(w, \mathfrak b)_{-\alpha_0}=0$.
  Hence, we have $w^{-1}(\alpha_0)<0$. 
 
 Now we prove the converse.
 
 Let $\psi_w:B\longrightarrow Aut^0(Z(w, \underline i))$ be the homomorphism of algebraic groups 
 induced by the action of $B$ on $Z(w, \underline i)$(as in the proof of Lemma \ref{l3}).
 Let $K$ be the kernel of $\psi_w$.
 Since $$BwB/B=\prod_{\beta \in R^+(w)}U_{-\beta}wB/B$$ (see \cite[Section 13.1]{Jan}) and 
 $w^{-1}(\alpha_0)<0$, we have $$U_{-\alpha_0}wB/B\neq wB/B.$$
 
 Since the desingularization map $\phi_w: Z(w, \underline i)\longrightarrow X(w)$ is 
 $B$-equivariant and the restriction of $\phi_w$ to an open subset is an isomorphism onto $BwB/B$,
   we have  $U_{-\alpha_0}\cap  K=\{e\}$, where $e$ is identity element in $B$.

   Recall that $f_w:\mathfrak b\longrightarrow H^0(Z(w, \underline i), T_{(w, \underline i)})$ is
   the homomorphism of Lie algebras induced by $\psi_w$.
 Since $U_{-\alpha_0}\cap K=\{e\}$, we have  $$(Ker (f_w))_{-\alpha_0}=0.$$
Since $Ker (f_w)$ is a $B$-submodule of $\mathfrak b$  and $\mathfrak b$ has a unique 
$B$-stable line $\mathfrak g_{-\alpha_0}$, we have  $Ker (f_w) =0$. 
 Hence $f_w$ is injective.  
  \end{proof}

The following  proposition describes the set of all positive roots occurring as a weight in 
$H^0(Z(w,\underline i), T_{(w,\underline i)})$.

\begin{proposition} \label{prop2} Let $w \in W$ and fix a reduced expression 
$w=s_{{i_1}}s_{{i_2}}\cdots s_{{i_r}}$.
 Let $\mu \in M_{\geq 0}\setminus \{0\}$. Then, we have 
 \begin{enumerate} 
 \item
    $H^0(Z(w,\underline i), T_{(w,\underline i)})_{\mu}\neq 0$ if and only if  there exists an integer $1\leq j \leq r$ 
 such that $\langle \alpha_{i_j}, \alpha_{i_k} \rangle = 0$ for all $1\leq k \leq j-1$, 
 and $\mu=\alpha_{i_j}.$ 

 \item Fix $1 \leq j\leq r$ such that $\langle \alpha_{i_j}, \alpha_{i_k} \rangle = 0$ for all $1\leq k \leq j-1$. Then, we have $dim (H^0(Z(w,\underline i), T_{(w,\underline i)})_{\alpha_{i_j}})=1$ 
 and $sl_{2, \alpha_{i_j}}$ is a $B_{\alpha_{i_j}}$-submodule of $H^0(Z(w,\underline i), T_{(w,\underline i)})$.
\end{enumerate}
 \end{proposition}
\begin{proof}

Proof of (1): Assume that $H^0(Z(w,\underline i), T_{(w,\underline i)})_{\mu}\neq 0$.
Let $v=s_{{i_1}}s_{{i_2}}\cdots s_{{i_{r-1}}}$ and let $i'=(i_1,i_2, \ldots, i_{r-1})$.
 By using $\it{LES}$ and Lemma \ref{vanishing}, we have the following exact sequence of 
 $B$-modules:\\
 $0\longrightarrow H^0(w, \alpha_{i_r})\longrightarrow H^0(Z(w,\underline i), T_{(w,\underline i)})
 \longrightarrow H^0(Z(v,\underline i'), T_{(v, \underline i')})\longrightarrow 
 H^1(w, \alpha_{i_r}) 
 \longrightarrow \\ H^1(Z(w,\underline i), T_{(w, \underline i)})\longrightarrow H^1(Z(v, 
 \underline i'), T_{(v,\underline i')})
 \longrightarrow 0$.
 
 Since $H^0(Z(w,\underline i), T_{(w,\underline i)})_{\mu}\neq 0$, either 
 $H^0(w, \alpha_{i_r})_{\mu}\neq 0$ or $H^0(Z(v, \underline i'), T_{(v,\underline i')})_
 {\mu}\neq 0$.

Now, if $H^0(w, \alpha_{i_r})_{\mu}\neq 0$, then by Corollary \ref{cor1}, we are done.

Otherwise, we have $H^0(Z(v,\underline i'), T_{(v,\underline i')})_{\mu}\neq 0$. 
Then by the induction on $l(w)$, there exists $1\leq j \leq r-1$ 
 such that $\langle \alpha_{i_j}, \alpha_{i_k} \rangle = 0$ for all $1\leq k \leq j-1$ and
 $\mu=\alpha_{i_j}.$ 
 
 We now prove the other implication: 
 
Let $1\leq j\leq r$ be such that $\langle \alpha_{i_j}, \alpha_{i_k} \rangle = 0$ for all 
$1\leq k \leq j-1$.

   If $j=r$, then $\langle \alpha_{i_k}, \alpha_{i_r} \rangle =0$ for all $1 \leq k\leq r-1$. 
   By Corollary \ref{cor1}, we have
 $$H^0(w, \alpha_{i_r})_{\alpha_{i_r}}\neq 0.$$ 
 Hence, we conclude that  $$H^0(Z(w,\underline i), T_{(w,\underline i)})_{\alpha_{i_r}}\neq 0.$$
 
  Otherwise, by Corollary \ref{cor1}, we have $H^0(w, \alpha_{i_r})_ {\alpha_{i_j}}=0$ and by 
  Corollary \ref{lemma3},
 we have $ H^1(w, \alpha_{i_{r}})_{\alpha_{i_j}}=0$. By the above exact sequence, 
we get $$H^0(Z(w,\underline i), T_{(w,i)})_{\alpha_{i_j}}\simeq H^0(Z(v,\underline i'), 
T_{(v, \underline i')})_{\alpha_{i_j}}.$$
Now the proof follows by induction on $l(w)$.

 Proof of (2):  
 Fix $1\leq j \leq r$.  Assume that  $\langle \alpha_{i_j}, \alpha_{i_k} \rangle = 0$ for all $1\leq k \leq j-1$. 
Then, by (1), we have  $H^0(Z(w, \underline i), T_{(w, \underline i)})_{\alpha_{i_j}}\neq 0$.

Let $v=s_{{i_1}}s_{{i_2}}\cdots s_{{i_{r-1}}}$ and $i'=(i_1,i_2, \ldots, i_{r-1})$.

If  $j=r$, then by Corollary \ref{cor1} we have $H^0(w, \alpha_{i_r}) \simeq sl_{2, \alpha_{i_r}}$.
Also, by using (1), we see that $H^0(Z(v, \underline i'), T_{(v, \underline i')})_
{\alpha_{i_r}}=0$.
Hence, by the above exact sequence, we conclude that $dim(H^0(Z(w, \underline i), T_{(w, \underline i)})_{\alpha_{i_r}})=1$
and $sl_{2, \alpha_{i_r}}$ is a $B_{\alpha_{i_r}}$-submodule of $H^0(Z(w,\underline i), T_{(w,\underline i)})$.

On the other hand, if $j\neq r$ then by induction on $l(w)$, $$dim(H^0(Z(v, \underline i'), T_{(v, \underline i')})_
{\alpha_{i_j}})=1$$ and $sl_{2, \alpha_{i_j}}$ is a $B_{\alpha_{i_j}}$-submodule of $H^0(Z(v,\underline i'), T_{(v,\underline i')})$.
Note that by Corollary \ref{cor1}, we have $H^0(w, \alpha_{i_r})_{\alpha_{i_j}}=0$.
Also, by Corollary \ref{lemma3}, we have $H^1(w, \alpha_{i_r})_{\alpha_{i_j}}=0$. 
Hence, by the above exact sequence, we see that
 $$H^0(Z(w, \underline i), T_{(w, \underline i)})_{\alpha_{i_j}}\simeq
H^0(Z(v, \underline i'), T_{(v, \underline i')})_{\alpha_{i_j}}$$ and
$dim(H^0(Z(w, \underline i), T_{(w, \underline i)})_{\alpha_{i_j}})=1.$
 
 
Further, since $sl_{2, \alpha_{i_j}}$ is a cyclic $B_{\alpha_{i_j}}$-module generated by $x_{\alpha_{i_j}}$, it follows that 
$x_{\alpha_{i_j}}$ is in the image of the map $H^0(Z(w, \underline i), T_{(w, \underline i)}) \longrightarrow
H^0(Z(v, \underline i'), T_{(v, \underline i')})$. Thus, we conclude  that $sl_{2, \alpha_{i_j}}$ is a $B_{\alpha_{i_j}}$-submodule of 
$H^0(Z(w, \underline i), T_{(w, \underline i)})$. 
\end{proof}

 \begin{proposition} \label{Borel1} Let $w\in W$ and $w=s_{{i_1}}s_{{i_2}}\cdots s_{{i_r}}$
be a reduced expression $\underline i$ of $w$. Then,  $H^0(Z(w, \underline i), T_{(w, \underline i)})$ contains a Lie subalgebra $\mathfrak b'$
  isomorphic to $\mathfrak b$ if and only if $w^{-1}(\alpha_{0})<0$.
 \end{proposition}
 \begin{proof}
Recall from the proof of Lemma \ref{Borel}, 
$\psi_w: B\longrightarrow Aut^0(Z(w, \underline i))$ is the homomorphism of algebraic groups 
induced by the action of $B$ on $Z(w, \underline i)$ 
and $f_w:\mathfrak b \longrightarrow H^0(Z(w, \underline i), T_{(w, \underline i)})$ is the 
induced homomorphism of Lie algebras. 

Assume that $\mathfrak b'$ is a Lie subalgebra of $H^0(Z(w, \underline i), 
T_{(w, \underline i)})$ 
which is isomorphic to $\mathfrak b$, then
there exists a closed subgroup $B'$ of $Aut^0(Z(w, \underline i))$ such that $B'$ is 
isomorphic to $B$ and $Lie(B')=\mathfrak b'$.

Fix an isomorphism $g:B \longrightarrow B'$. Then, $g(T)(\simeq T)$ is a maximal torus in $B'$.
Hence, we have  $$rank(Aut^0(Z(w, \underline i)))\geq dim (T).$$
By Proposition \ref{Prop0}(3), $f_w(\mathfrak h)$ is a maximal toral subalgebra of $H^0(Z(w, \underline i),
T_{(w, \underline i)})$. Hence, $\psi_w (T)$ is a maximal torus in $Aut^0(Z(w, \underline i))$.
Thus, the restriction $\psi_w|_{T}: T\longrightarrow Aut^0(Z(w, \underline i))$ is injective.

Let $T'$ be a maximal torus of $B'$. Since any two maximal tori in $Aut^0(Z(w, \underline i))$ 
are conjugate, there exists a $\sigma \in Aut^0(Z(w, \underline i))$ such that 
$T =\sigma T' \sigma^{-1}$. Now, let $B'':=\sigma B' \sigma^{-1}$. Then, we have  $T\subset B''$.
Since $Lie(B'')$ is a $T$-stable Lie subalgebra of $H^0(Z(w, \underline i), T_{(w, \underline i)})$
, by Proposition \ref{prop2} we have $$Lie(B'')=\mathfrak h\oplus\bigoplus_{\beta \in R' }\mathfrak 
g_{\beta}\oplus\bigoplus_{\alpha \in S'}\mathfrak g_{\alpha}$$ for some subset $R'$ of $R^-$ and for 
some subset $S'$ of $S$. 

 Fix $\alpha\in S'$, Then, we have $H^0(Z(w, \underline i), T_{(w, \underline i)})_{
 \alpha}\neq 0$. Hence by Proposition \ref{prop2}, we have \\$dim(H^0(Z(w, \underline i), 
 T_{(w, \underline i)})_{\alpha})=1$.
 Thus, the homomorphism $f_w:\mathfrak b\longrightarrow H^0(Z(w, \underline i),
 T_{(w, \underline i)})$ extends to $\widetilde{f_w}: \mathfrak p_{\alpha}\longrightarrow 
 H^0(Z(w, \underline i), T_{(w, \underline i)})$ as $T$-modules such that $\widetilde{f_w}
 (\mathfrak g_{\alpha})\neq 0$.
 Let $\mathfrak l_{\alpha}\subseteq \mathfrak p_{\alpha}$ be the Lie algebra of $L_{\alpha}$. 
 Consider the restriction $(f_w)_{\alpha}$ of $\widetilde{f_w}$ to $\mathfrak l_{\alpha}$.
 Clearly, $(f_w)_{\alpha}$ is injective homomorphism of Lie algebras.
 Let $n_{\alpha}$ be a representative of the simple reflection $s_{\alpha}$ in $N_{G}(T)$, let 
 $(\psi_w)_{\alpha}:\widetilde L_{\alpha}\longrightarrow Aut^0(Z(w, \underline i))$ be the 
 homomorphism of algebraic groups induced by $f_w{_{\alpha}}$, where $\widetilde L_{\alpha}$ 
 is a simply connected covering of $L_{\alpha}$.
 Since $(f_{w})_{\alpha}$ is injective, $\widetilde n_{\alpha}\notin Ker ((\psi_w)_{\alpha})$, where
 $\widetilde n_{\alpha}$ is a lift of $n_{\alpha}$ in $\widetilde L_{\alpha}$.
 Note that $(\psi_{w})_{\alpha}(n_{\alpha})$ normalizes $T$ and hence $Ad((\psi_w)_{\alpha}(n_{\alpha}))
 (\mathfrak h)=\mathfrak h$.

 Since $Lie(B'')$ is solvable Lie subalgebra  and $\mathfrak g_{\alpha}\subseteq Lie(B'')$, 
 $\mathfrak g_{-\alpha}\nsubseteq Lie(B'')$ (otherwise, $sl_{2, \alpha}$ would be Lie 
 subalgebra of $Lie(B'')$).
Hence, we have $R'\cap (-S')=\emptyset$.

Note that by Proposition \ref{prop2}, if $\alpha\in S'$, then $\alpha=\alpha_{i_j}$ for some  
integer $1\leq j \leq r$ such that  $\langle \alpha_{i_j}, \alpha_{i_k} \rangle = 0$ for all 
$1\leq k \leq j-1$.
Hence,  the elements in $\{ s_{\alpha}: \alpha\in S' \}$ commute with each other.
Thus,  $(\prod_{\alpha\in S'}s_{\alpha})(\beta)=-\beta$ for every $\beta \in S'$. Further,
  since $R'\cap (-S')=\emptyset$, we have  $(\prod_{\alpha\in S'}s_{\alpha})(R')\subseteq R^{-}$.
  Let $n=\prod_{\alpha\in S'}(\psi_w)_{\alpha}(\widetilde n_{\alpha})$, where the product 
  is taken in some ordering.
Hence $$Lie(nB''n^{-1})=\mathfrak h\oplus\bigoplus_{\beta \in R''} \mathfrak g_{\beta}\oplus\bigoplus_
{\gamma\in S'} \mathfrak g_{\gamma},$$
where $R''=(\prod_{\alpha\in S'}s_{\alpha})(R')$. 
Note that for each $\alpha \in S'$, $s_{\alpha}(R')\cap(-S')=\emptyset$. 
Hence $R''\cap (-S')=\emptyset$.
Then,  $Lie(nB''n^{-1})\subseteq \mathfrak b$. 
Since $dim(\mathfrak b)=dim(Lie(nB''n^{-1}))$, we have $Lie(nB''n^{-1})=\mathfrak b$.

In particular, we have $H^0(Z(w, \underline i), T_{(w, \underline i)})_{-\alpha_0}\neq 0$. 
Since $H^0(Z(w, \underline i), T_{(w, \underline i)})$ is a $B$-submodule of $H^0(w, \mathfrak 
g/\mathfrak b)$, we have $H^0(w, \mathfrak g/\mathfrak b)_{-\alpha_0}\neq 0$.
Hence, we have $w^{-1}(\alpha_0)<0$.

 Proof of the converse follows from Lemma \ref{Borel}.
\end{proof}

\section{Automorphism group of $Z(w, \underline i)$:}

In this section, we study the automorphism group of a BSDH variety. 

Let $w\in W$ and fix a reduced expression $w=s_{{i_1}}s_{{i_2}}\cdots s_{{i_r}}$, 
let $\underline i =(i_1, i_2, \ldots, i_r).$

Recall that for any reduced expression $w_0=s_{{j_1}}s_{{j_2}}\cdots s_{{j_N}}$ of $w_0$ such that 
$\underline j=(j_1, j_2, \ldots, j_N)$ and $(j_1, j_2, \ldots, j_r)=\underline i$, there exits a 
natural homomorphism $$f_w: H^0(Z(w_0, \underline j), T_{(w_0, \underline j)})
\longrightarrow H^0(Z(w, \underline i), T_{(w, \underline i)})$$ of Lie algebras from Section $5$.

 Recall the following notation:
$$J^{'}(w, \underline i):=\{l\in \{1,2, \ldots, r\}:\langle \alpha_{i_l}, \alpha_{i_k} 
\rangle = 0~{\rm 
for ~all }~ k < l\}$$ $$J(w, \underline i):=\{\alpha_{i_l}: l\in J^{'}(w, \underline i)\}\subset 
S.$$ Note that
 the simple reflections $\{s_{i_j}: j\in J'(w, \underline i)\}$ commute with each other.
For each $\alpha$ in $J(w, \underline i)$, fix a representative $n_{\alpha}$ of $s_{\alpha}$ in 
$N_G(T)$ and let $P_{J(w, \underline i)}$
be the subgroup of $G$ generated by $B$ and $\{n_{\alpha}: \alpha \in J(w, \underline i)\}$.
Let $\mathfrak p_{J(w, \underline i)}$ be the Lie algebra of $P_{J(w, \underline i)}$.

Then, we have 

\begin{theorem}\label{theorem1}{\  }
\begin{enumerate} 
\item
 $\mathfrak p_{J(w_0, \underline i)}\simeq H^0(Z(w_0, \underline i), T_{(w_0, \underline i)})$.
 \item
  $\mathfrak p_{J(w, \underline i)}$ is isomorphic to a Lie subalgebra of  
  $H^0(Z(w, \underline i), T_{(w, \underline i)})$ if and only if $ w^{-1}(\alpha_0)<0$.
  In such a case, we have $\mathfrak p_{J(w, \underline i)}= \mathfrak p_{J(w_0, \underline j)}$
  for any reduced expression $w_0=s_{{j_1}}s_{{j_2}}\cdots s_{{j_N}}$ of $w_0$ such that 
$\underline j=(j_1, j_2, \ldots, j_N)$ and $(j_1, j_2, \ldots, j_r)=\underline i$.
\item
 If $G$ is simply laced, $\mathfrak p_{J(w, \underline i)}\simeq H^0(Z(w, \underline i), 
 T_{(w, \underline i)})$ if and only if $ w^{-1}(\alpha_0)<0$.
 In such a case, we have $\mathfrak p_{J(w_0, \underline j)}\simeq H^0(Z(w, \underline i), 
 T_{(w, \underline i)})$, where $\underline j$ is as in (2).
\item
 If $G$ is simply laced, $f_w: H^0(Z(w_0, \underline j), T_{(w_0, \underline j)})
\longrightarrow H^0(Z(w, \underline i), T_{(w, \underline i)})$ is surjective,
where $\underline j$ is as in (2).
 
 
\end{enumerate}
 \end{theorem}
\begin{proof}


Proof of (1): 
 By Lemma \ref{l3}(2),  $f_{w_0}:\mathfrak b\longrightarrow H^0(Z(w_0, \underline i), T_{(w_0, \underline i)})$ is 
 injective. Also, by Corollary \ref{C1}(1), $H^0(Z(w_0, \underline i), T_{(w_0, \underline i)})$ is Lie subalgebra of $\mathfrak g$.

By Proposition \ref{prop2},
any $\mu\in M_{\geq 0}\setminus\{0\}$ such that 
 $H^0(Z(w_0,\underline i), T_{(w_0,\underline i)})_{\mu}\neq 0$ is of the form 
  $\mu=\alpha_{i_j}$ for some $1\leq j\leq r$
 such that $\langle \alpha_{i_j}, \alpha_{i_k} \rangle = 0$ for all $1\leq k \leq j-1$.
Hence,
we conclude that  $H^0(Z(w_0, \underline i), T_{(w_0, \underline i)})$
is isomorphic to $\mathfrak p_{J(w_0, \underline i)}$.

Proof of (2): If 
$\mathfrak p_{J(w, \underline i)}$ is isomorphic to a Lie subalgebra of  $H^0(Z(w, \underline i),
T_{(w, \underline i)})$, then by Proposition \ref{Borel1}, we have $w^{-1}(\alpha_0)<0$.

Conversely, assume that $w^{-1}(\alpha_0)<0$.
Let  $w_0=s_{{j_1}}s_{{j_2}}\cdots s_{{j_N}}$ be a  reduced expression of $w_0$
such that $\underline i =(j_1, j_2, \ldots, j_r)$. Set $\underline j=(j_1,j_2,\ldots,j_N)$.
Clearly, $J(w, \underline i)\subset J(w_0, \underline j)$. Hence, we have $\mathfrak p_{J(w, \underline i)}
\subset \mathfrak p_{J(w_0, \underline j)}$.

Therefore, by using (1), $\mathfrak p_{J(w, \underline i)}$ is a Lie subalgebra of $H^0(Z(w_0, \underline j), T_{(w_0, \underline j)})$.

Now, recall the following commutative diagram of Lie algebras:

 \begin{center}  
  $\xymatrix{\mathfrak p_{J(w, \underline i)}  \ar@{<-_)}[d]  \ar@{^{(}->}[rr] && H^{0}(Z(w_0, \underline j) , T_{(w_0 , \underline j)}) \ar[d]^{f_w} \ar@{^{(}->}[r] &  \mathfrak g \\
\mathfrak b \ar[rr]^{f_w|_{\mathfrak b}} && H^0(Z(w, \underline i), T_{(w, \underline i)}) 
} $
\end{center}
(see Section 5).

  Since the unique $B$-stable line $\mathfrak g_{-\alpha_0}$
 in $H^{0}(Z(w_0, \underline j) , T_{(w_0 , \underline j)})$ lies in $\mathfrak b$, 
 by commutativity of the above diagram, we conclude that
  $f_w:H^{0}(Z(w_0, \underline j) , T_{(w_0 , \underline j)}) \longrightarrow H^0(Z(w, \underline i),
T_{(w, \underline i)}) $ is injective if and only if its restriction $f_w|_{\mathfrak b}$ to $\mathfrak b$ is injective.

Since $w^{-1}(\alpha_0)<0$, by  Lemma \ref{Borel}, $f_w|_{\mathfrak b}$ to $\mathfrak b$ is injective. Hence, by the above arguments, 
$$f_w:H^{0}(Z(w_0, \underline j) , T_{(w_0 , \underline j)}) \longrightarrow H^0(Z(w, \underline i),
T_{(w, \underline i)}) $$ is injective. Therefore,  $H^0(Z(w, \underline i),
T_{(w, \underline i)})_{\alpha}\neq 0$ for every $\alpha \in J(w_0, \underline j)$. Thus, we conclude that
$J(w_0, \underline j)=J(w, \underline i)$.

Proof of (3): 
If $G$ is  simply laced, 
by Theorem \ref{AS} (3), we have $H^0(w, \mathfrak g/\mathfrak b)=\mathfrak g$ if and only if 
$w^{-1}(\alpha_0)<0$. 
Recall from Section $5$ that $H^0(Z(w, \underline i), T_{(w, \underline i)})$ is a $B$-submodule of
$H^0(w,\mathfrak g/\mathfrak b )$. 
Hence, from the proof of (2), we conclude that $\mathfrak p_{J(w, \underline i)} \simeq 
H^0(Z(w, \underline i), T_{(w, \underline i)})$ if and only if $w^{-1}(\alpha_0)<0$. 

Proof of (4):  Proof is by descending induction on $l(w)$. 
If $w=w_0$, we are done. Otherwise, let $w_0=s_{{j_1}}s_{{j_2}}\cdots s_{{j_N}}$
 be  a reduced expression for $w_0$ such that $(j_1,j_2, \ldots, j_r)=\underline i$ and 
 $r\leq N-1$.
Let $v=s_{{j_1}}s_{{j_2}}\cdots s_{{j_{r+1}}}$ and let $\underline i'=(j_1, j_2, \ldots, 
j_{r+1})$. Note that $l(w)=l(v)-1$.
 
Since $G$ is simply laced, by using $\it{LES}$ and Lemma \ref{vanishing} (2) we have the 
following short exact sequence of $B$-modules:
 $$0\longrightarrow H^0(v, \alpha_{i_{r+1}})\longrightarrow H^0(Z(v, \underline i'), 
 T_{(v,\underline i')})\longrightarrow H^0(Z(w, \underline i), T_{(w, \underline i)})
 \longrightarrow H^1(v, \alpha_{i_{r+1}})=0.$$
 
 Consider the following commutative diagram of Lie algebras:
 \begin{center}  
  $\xymatrix{H^{0}(Z(w_0, \underline j) , T_{(w_0 , \underline j)}) \ar[dr]^{f_w} \ar[d]^{f_v}\\
   H^{0}(Z(v, \underline i') , T_{(v , \underline i')})
\ar[r] & 
H^0(Z(w, \underline i), 
T_{(w, \underline i)}) } $
\end{center}

 By  descending induction on $l(w)$, 
 $f_v: H^0(Z(w_0, \underline j), 
 T_{(w_0, \underline j)}) \longrightarrow H^0(Z(v, \underline i'), T_{(v, \underline i')})$ is
 surjective.
 By commutativity of the above diagram and by the above short exact sequence, we conclude that $f_w:H^0(Z(w_0, \underline j), T_{(w_0, \underline j)}) \longrightarrow H^0(Z(w, \underline i), T_{(w, \underline 
 i)})$ is surjective.
This completes the proof of (4).
\end{proof}



Recall that $\leq$ is the Bruhat-Chevalley ordering on $W$
and $supp(w):=\{j\in \{1,2,\ldots, n\} : s_j\leq w\}$, the support of $w$. For simplicity of notation 
we denote $supp(w)$ by $A_w$. For $j\in A_w$, let $n_j$ be a representative of 
$s_j$ in $N_G(T)$.
Let $P_{A_w}$ be the standard parabolic subgroup of $G$ containing $B$ and $\{n_j: j\in A_w\}.$ 
Let $\mathfrak p_{A_w}$ be the Lie algebra of $P_{A_w}$. 

Let $w=s_{{i_1}}s_{{i_2}}\cdots s_{{i_r}}$ be a reduced expression of $w$ and let $\underline i =(i_1, i_2, \ldots, i_r).$
Let $w_0=s_{{j_1}}s_{{j_2}}\cdots s_{{j_N}}$ be a reduced expression for $w_0$
such that $(j_{1}, j_{2}, \ldots, j_{r})=\underline i.$

Set $J_1:=(\{1,2, \ldots, n\}\setminus A_w)\cap J'(w_0, \underline j)$.
Let $R_w=R^+\setminus (\bigcup_{v\leq w}R^{+}(v^{-1}))$.

Let $f_w:H^0(Z(w_0, \underline j), T_{(w_0, \underline j)}) \longrightarrow H^0(Z(w, \underline i), T_{(w, \underline 
 i)})$ be the homomorphism as above.

Now, we will describe the kernel of the  map $f_w$ when $G$ is simply laced.
Let $Ker (f_w)$ be the kernel of $f_w$.

\begin{corollary} \label{kernel} Let $G$ be simply laced. Then, we have 
$$Ker (f_w) = (\bigcap_{k\in A_w} Ker (\alpha_k)) \oplus (\bigoplus_{\beta\in R_w}\mathfrak g_{-\beta})\oplus (\bigoplus_{j\in J_1} \mathfrak g_{\alpha_j}).$$

\end{corollary}
\begin{proof}
Step 1: We will prove that for every $j\in A_w$, the restriction of $f_{w}$
to the subspace $\mathbb C.h_{\alpha_j}\oplus \mathfrak g_{-\alpha_j}$ is injective.

Fix $j\in A_w$. 
Let $k$ be the least positive integer in $\{1,2,\ldots, r\}$ such that $j=i_k$.
Let $v=s_{i_1}s_{i_2}\cdots s_{i_k}$ and set $\underline i'=(i_1,\ldots, i_k)$.
Then, by Lemma \ref{1}(2), we see that $\mathbb C.h_{\alpha_j}\oplus \mathfrak g_{-\alpha_j}$
is a $B_{\alpha_j}$-submodule of $H^0(v, \alpha_j)$.
By {\it LES}, $H^0(v, \alpha_j)$ is a $B$-submodule of $H^{0}(Z(v, \underline i') , T_{(v , \underline i')}).$
Let $g:H^0(Z(w, \underline i), 
T_{(w, \underline i)})\longrightarrow H^{0}(Z(v, \underline i') , T_{(v , \underline i')})$ be the homomorphism of 
$B$-modules induced by the fibration $Z(w, \underline i)\longrightarrow Z(v, \underline i')$.

Now, consider  the following commutative diagram of $B$-modules:
\begin{center}  
  $\xymatrix{H^{0}(Z(w_0, \underline j) , T_{(w_0 , \underline j)}) \ar[dr]^{f_v} \ar[d]^{f_w}\\
H^0(Z(w, \underline i), 
T_{(w, \underline i)}) 
\ar[r]^{g} &  H^{0}(Z(v, \underline i') , T_{(v , \underline i')})} $
\end{center}

Note that $\mathbb C.h_{\alpha_j}\oplus \mathfrak g_{-\alpha_j}$ is a subspace of $H^{0}(Z(w_0, \underline j) , T_{(w_0 , \underline j)})$.
Therefore, by the above arguments, the restriction of $f_{v}$ to the subspace $\mathbb C.h_{\alpha_j}\oplus \mathfrak g_{-\alpha_j}$ is injective.
Hence, by  commutativity of the above diagram, we conclude that 
the restriction of $f_{w}$ to the subspace $\mathbb C.h_{\alpha_j}\oplus \mathfrak g_{-\alpha_j}$ is injective.


Step 2: Let $\mathfrak l_{A_w}$ be the Levi subalgebra of $\mathfrak p_{A_w}$, let $\mathfrak z(\mathfrak l_{A_w})$
be the center of $\mathfrak l_{A_w}$. 
We will prove that $$\mathfrak h \cap Ker (f_w) =\mathfrak z(\mathfrak l_{A_w})=\bigcap_{k\in A_w} Ker (\alpha_k).$$

First note that $\bigcap_{k\in A_w} Ker (\alpha_k)= \mathfrak z(\mathfrak l_{A_w})$ and the dimension of 
$\mathfrak z(\mathfrak l_{A_w})$ is $n-d(w)$ (since $|A_w|=d(w)$).  

 Now, we prove that $\mathfrak h \cap Ker (f_w)$ is contained in $\bigcap_{k\in A_w} Ker (\alpha_k) $.
 
 Assume the contrary. Then, there exists a $k\in A_w$ and $h\in \mathfrak h \cap Ker (f_w) $ such that 
 $\alpha_k(h)\neq 0$. Then,  $$x_{-\alpha_k}\cdot h= -[h, x_{-\alpha_k}]=\alpha_k(h)x_{-\alpha_k}$$ is a non zero multiple of 
 $x_{-\alpha_k}$. Hence $\mathfrak g_{-\alpha_k}$ is contained in $Ker (f_w)$, which 
  contradicts step 1.
 Therefore, $\mathfrak h \cap Ker (f_w)$ is contained in $\bigcap_{k\in A_w} Ker (\alpha_k) $.

By Proposition \ref{Prop0}, we have $H^{0}(Z(w_{0}, \underline j) , T_{(w_{0} , \underline j)})_0=\mathfrak h$  and $dim(\mathfrak h \cap Ker (f_w))=n-d(w)$. 
Hence, we see that $$f_w(H^{0}(Z(w_{0}, \underline j) , T_{(w_{0} , \underline j)})_0)=
H^0(Z(w, \underline i), T_{(w, \underline i)})_0.$$

By the above arguments, $\mathfrak h \cap Ker (f_w)$ is a subspace of $\bigcap_{k\in A_w} Ker (\alpha_k) $
having the same dimension as that of $\bigcap_{k\in A_w} Ker (\alpha_k) $.
Hence, we conclude that 
$$\mathfrak h \cap Ker (f_w) = \bigcap_{k\in A_w} Ker (\alpha_k)=\mathfrak z(\mathfrak l_{A_w}).$$

 Step 3: We will prove that for $j\in J_1$, $sl_{2, \alpha_j}$ is contained in $Ker (f_w).$
 
 
 Fix $j\in J_1$.  By Theorem \ref{theorem1}(2), it follows that $sl_{2, \alpha_j}$ is a $B_{\alpha_j}$-submodule of
\\$H^{0}(Z(w_{0}, \underline j) , T_{(w_{0} , \underline j)}).$
By Proposition \ref{prop2}(1),
we see that
$H^0(Z(w, \underline i), T_{(w, \underline i)})_{\alpha_j}=0$. 
Hence, $\mathfrak g_{\alpha_j} \subset Ker (f_w)$.
Since $sl_{2, \alpha_j}$ is a cyclic $B_{\alpha_j}$-module generated by
$\mathfrak g_{\alpha_j}$, it 
follows that $sl_{2, \alpha_j}$ is contained in $Ker (f_w).$

Step 4: The intersection of the nilradical of $\mathfrak b$ and  
$Ker (f_w)$ is equal to the direct sum  $\bigoplus_{\beta\in R_w}\mathfrak g_{-\beta}$ of $T$-modules.

Consider the birational morphism $\phi_w:Z(w, \underline i)\longrightarrow X(w)$.
Note that $\phi_w$ is a $B$-equivariant morphism for the natural left action of $B$ on $Z(w, \underline i)$ 
(respectively, on $X(w)$).
Let $\phi:B\longrightarrow Aut^0(X(w))$ (respectively, $\phi':B\longrightarrow Aut^0(Z(w, \underline i))$)
be the homomorphism induced by the action of $B$ on $X(w)$ (respectively, on $Z(w, \underline i)$).
Since $\phi_w$ is birational, we have $Ker(\phi)\cap B_u=Ker(\phi')\cap B_u$, where $B_u$ is the unipotent radical of $B$.

Since $G$ is simply laced, by \cite[Corollary 3.9]{Ka4}, we conclude that $\mathfrak b_u \cap Ker(f_w)=\bigoplus_{\beta\in R_w}\mathfrak g_{-\beta}$, where
$\mathfrak b_u$ is the nilradical of $\mathfrak b$.

From the steps 1 to 4, we conclude that    
$$Ker (f_w) = (\bigcap_{k\in A_w} Ker (\alpha_k)) \oplus (\bigoplus_{\beta\in R_w}\mathfrak g_{-\beta})\oplus (\bigoplus_{j\in J_1} \mathfrak g_{\alpha_j}).$$

\end{proof}

Recall that if $X$ is a smooth projective variety over $\mathbb C$, the connected component 
of the group of all automorphisms of $X$ containing identity automorphism is an algebraic group
(see \cite[p.17, Theorem 3.7]{Mat}, \cite[p.268]{Grothendieck}, which also deals the case when 
$X$ may be singular or it may be defined over any field). Futher, the Lie algebrs of this automorphism
group is isomorphic to the space of all vector fields on $X$, that is the space $H^0(X, T_X)$ of all
global sections of the tangent bundle $T_X$ of $X$ (see \cite[p.13, Lemma 3.4]{Mat}).

We now prove the main results of the paper using Theorem \ref{theorem1}.

  Recall that $Aut^0(Z(w, \underline i))$ is the connected component of the identity element of 
  the automorphism group of $Z(w, \underline i)$. 
  
      \begin{theorem}\label{cor3}{\ }
\begin{enumerate}
\item $ P_{J(w_0, \underline i)} \simeq Aut^0(Z(w_0, \underline i))$.
  \item $Aut^0(Z(w, \underline i))$ contains a closed subgroup isomorphic to $ 
 P_{J(w, \underline i)}$   if and only if $ w^{-1}(\alpha_0)<0$. In such a case, we have 
  $P_{J(w, \underline i)} = P_{J(w_0, \underline j)}$  for any reduced expression $w_0=s_{{j_1}}s_{{j_2}}\cdots s_{{j_N}}$ of $w_0$ such that 
$\underline j=(j_1, j_2, \cdots, j_N)$ and $(j_1, j_2, \cdots, j_r)=\underline i$.
 \item If $G$ is simply laced,  
  $ P_{J(w, \underline i)}\simeq Aut^0(Z(w, \underline i)) $ if and only if $ w^{-1}(\alpha_0)<0$.
 In such a case, we have $Aut^0(Z(w, \underline i)) \simeq Aut^0(Z(w_0, \underline j))$, where $\underline j$
 is as in (2).
 \item  The homomorphism 
 $f_w: H^{0}(Z(w_{0}, \underline j) , T_{(w_{0} , \underline j)})\longrightarrow H^{0}(Z(w, \underline i) , T_{(w , \underline i)})$
 is induced by a homomorphism $g_w : Aut^0(Z(w_0, \underline j)) \longrightarrow Aut^0(Z(w, \underline i))$
of algebraic groups, where $\underline j$ is as in (2). 
  \item If $G$ is simply laced, the homomorphism $g_w : Aut^0(Z(w_0, \underline j)) \longrightarrow Aut^0(Z(w, \underline i))$
of algebraic groups is surjective, where $\underline j$ is as in (2).
 \item The rank of $Aut^0(Z(w, \underline i))$ is at most the rank of $G$.
 \end{enumerate}
     \end{theorem}
\begin{proof} Recall that by \cite[Theorem 3.7]{Mat}, $Aut^0(Z(w, \underline i))$ is an 
algebraic group 
and  $$Lie(Aut^0(Z(w, \underline i)))=H^0(Z(w, \underline i), T_{(w, \underline i)}).$$

Let $\pi:\widetilde G \longrightarrow G$ be the simply connected covering of $G$.
Let $\widetilde P_{J(w, \underline i)}$ (respectively, $\widetilde B$) be the inverse image of 
$P_{J(w, \underline i)}$ (respectively, of $B$)  in $\widetilde G$.  

Proof of (2): 
 If $w^{-1}(\alpha_0)<0$, then by Theorem \ref{theorem1}(2), $\mathfrak p_{J(w, \underline i)}$ is 
 isomorphic to a Lie subalgebra of  $H^0(Z(w, \underline i), T_{(w, \underline i)})$. 
Hence, there is a homomorphism $\widetilde \psi_w:\widetilde P_{J(w, \underline i)}\longrightarrow 
Aut^0(Z(w,\underline i))$ of algebraic groups.
Since the center $Z(\widetilde P_{J(w, \underline i)})$ of $\widetilde P_{J(w, \underline i)}$ is 
same as $Z(\widetilde B)$ and $B$ acts on $Z(w,\underline i)$, 
$Z(\widetilde P_{J(w, \underline i)})$ acts trivially on $Z(w, \underline i)$. Hence, the action 
of $\widetilde P_{J(w, \underline i)}$ induces a
homomorphism $\psi_w: P_{J(w, \underline i)}\longrightarrow Aut^0(Z(w, \underline i))$ of algebraic 
groups.
Since $\mathfrak p_{J(w, \underline i)}$ is isomorphic to a Lie subalgebra of  
$H^0(Z(w, \underline i), T_{(w, \underline i)})$, $\psi_w$ is an isomorphism onto its image.

On the other hand, if $Aut^0(Z(w, \underline i))$ contains a closed subgroup isomorphic to $ P_{J(w, \underline i)}$,
then there is an injective homomorphism $\psi_w:P_{J(w, \underline i)}\longrightarrow 
Aut^0(Z(w, \underline i))$ of algebraic groups. Further, $\psi_w$ induces an
injective homomorphism $\widetilde f_w:\mathfrak p_{J(w, \underline i)}\longrightarrow 
H^0(Z(w, \underline i), T_{(w, \underline i)})$ of Lie algebras. Hence, by Theorem \ref{theorem1}(2),
we have $w^{-1}(\alpha_0)<0$.
This completes the proof of (2).

Proofs of (1), (3) and (4) are similar to the proof of (2). For the sake of completeness we give proof here.

Proof of (1).
By Theorem \ref{theorem1}(1), $\mathfrak p_{J(w_0, \underline i)}$ is 
 isomorphic to the Lie algebra $H^0(Z(w_0, \underline i), T_{(w_0, \underline i)})$. 
Hence, there is a homomorphism $\widetilde \psi_{w_0}:\widetilde P_{J(w_0, \underline i)}\longrightarrow 
Aut^0(Z(w_0,\underline i))$ of algebraic groups.
Since the center $Z(\widetilde P_{J(w_0, \underline i)})$ of $\widetilde P_{J(w_0, \underline i)}$ is 
same as $Z(\widetilde B)$ and $B$ acts on $Z(w_0,\underline i)$, 
$Z(\widetilde P_{J(w_0, \underline i)})$ acts trivially on $Z(w_0, \underline i)$. Hence, the action 
of $\widetilde P_{J(w_0, \underline i)}$ induces a
homomorphism $\psi_{w_0}: P_{J(w_0, \underline i)}\longrightarrow Aut^0(Z(w_0, \underline i))$ of algebraic 
groups. Note that  $\psi_{w_0}$ induces an isomorphism $\widetilde f_{w_0}:\mathfrak p_{J(w_0, \underline i)}\longrightarrow 
H^0(Z(w_0, \underline i), T_{(w_0, \underline i)})$ of Lie algebras.
Hence, we conclude that $\psi_{w_0}:P_{J(w_0, \underline i)}\longrightarrow 
Aut^0(Z(w_0,\underline i))$ is an isomorphism of algebraic groups.

Proof of (3). By (2), we have the homomorphism $\psi_w:P_{J(w, \underline i)}\longrightarrow 
Aut^0(Z(w, \underline i))$ of algebraic groups is injective if and only if $w^{-1}(\alpha_0)<0$.
Since $G$ is simply laced, by Theorem \ref{theorem1}(3), we conclude the proof of (3). 

Proof of (4). By (1), we have $P_{J(w_0, \underline j)} \simeq Aut^0(Z(w_0, \underline j))$.

Let $$P_{J(w_0, \underline j)}=LP_u=L_{ss}Z(L)P_u$$ be the Levi decomposition of $P_{J(w_0, \underline j)}$ such that $T\subset L$, where
$L$ is the Levi factor of $P_{J(w_0, \underline j)}$ containing $T$, $L_{ss}$ is semi simple part of $L$ and $P_u$ is 
unipotent radical of $P_{J(w_0, \underline j)}$.

Since $P_u\subset B$, we have the homomorphism $f_1: P_u\longrightarrow Aut^0(Z(w, \underline i))$ of algebraic groups.

Since $Z(L)\subset T\subset B$, we have the homomorphism $f_2: Z(L)\longrightarrow Aut^0(Z(w, \underline i))$.

For $j\in J(w, \underline i)$, by Lemma \ref{prop2}, $sl_{2, \alpha_j}$ is contained in 
$H^{0}(Z(w, \underline i), T_{(w, \underline i)})$.
Hence for each $j\in J(w, \underline i)$,  we have $\phi_{j}: SL_{2, \alpha_j}\longrightarrow Aut^0(Z(w, \underline i))$.

For $j\in J(w_0, \underline j)\setminus J(w, \underline i)$, by the proof of Corollary 
\ref{kernel} (even though $G$ is not necessarily simply laced), we have $\mathfrak g_{\alpha_j} \subset Ker(f_w)$.
Hence, the homomorphism $\phi_{j}: SL_{2, \alpha_j}\longrightarrow Aut^0(Z(w, \underline i))$ is trivial.
That is $SL_{2, \alpha_j}$ acts trivially on $Z(w, \underline i))$ for each $j\in J(w_0, \underline j)\setminus J(w, \underline i)$.

Therefore, we have the homomorphism $\widetilde L \longrightarrow Aut^0(Z(w, \underline i))$ of algebraic groups,
where $\widetilde L$ is inverse image of $L$ in $\widetilde G$ by the universal cover $\pi:\widetilde G\longrightarrow G$. 

Claim: For $j\in J(w_0, \underline j)$, we have the following commutative diagram of algebraic groups:

\begin{center}  
  $\xymatrix{ SL_{2, \alpha_j} \ar[d] \ar[r]^{\phi_j} & Aut^0(Z(w, \underline i))\\
\ar[ur] PGL_{2,\alpha_j}}$
\end{center}

Let $G_{\alpha_j}$ be the image of $SL_{2, \alpha_j}$ in $Aut^0(Z(w, \underline i))$,
let $B_{\alpha_j}=B\cap G_{\alpha_j}$. Let $\widetilde B_{\alpha_j}=\pi^{-1}(B_{\alpha_j})$, which is a Borel
subgroup of $SL_{2, \alpha_j}$.

Now consider the following commutative diagram:
\begin{center}  
  $\xymatrix{ 
    SL_{2, \alpha_j} \ar[r]^{\phi_j} & Aut^0(Z(w, \underline i))\\
 \widetilde B_{\alpha_j} \ar@{^{(}->}[u] \ar[r]^{\pi} & \ar@{^{(}->}[u] B_{\alpha_j}}$
\end{center}

Since the kernel of $\pi$ is contained in the kernel of $\phi_j$,  the action of $Z(\widetilde B_{\alpha_j})$ 
on 
$Z(w, \underline i)$ is trivial.
Since $Z(\widetilde B_{\alpha_j})=Z(SL_{2, \alpha_j})$, 
we have the homomorphism $PSL_{2, \alpha_j}\longrightarrow  Aut^0(Z(w, \underline i))$. This proves the claim.

From the above discussion, we conclude that 
 the center $Z(\widetilde P_{J(w_0, \underline j)})$ acts trivially on $Z(w, \underline i)$. 
Hence, there is a homomorphism $g_w : P_{J(w_0, \underline j)} \longrightarrow Aut^0(Z(w, \underline i))$
of algebraic groups which  induces $f_w$.  This completes the proof of (4).

Proof of (5) follows from Theorem \ref{theorem1}(4).

Proof of (6)
follows from Proposition \ref{Prop0}.
\end{proof}

We use the same notation as before. Assume that $G$ is simply laced.

Let $g_w: Aut^0(Z(w_0, \underline j))\longrightarrow Aut^0(Z(w, \underline i))$ be the natural map as in Theorem \ref{cor3} (4).
 Let $U^+$ be  the unipotent radical of $B^+$.
For $j\in J_1$, let $U_{\alpha_j}^+$ denote the one-dimensional $T$-stable closed subgroup of $U^+$ (for the conjugation action of 
$T$ on $G$) corresponding to $\alpha_{j}$. 
Let $T(w):=\bigcap_{k\in A_w}Ker(\alpha_k)$.
 Since $\{\alpha_k: k\in A_w\}$ is a subset of the $\mathbb Z$-basis $S$ of $X(T)$, $T(w)$ is connected.
\begin{corollary}\label{kernel1} The connected component of the
 kernel of the  map $g_w$ is the closed subgroup of 
 $Aut^0(Z(w_0, \underline j))$ generated by the torus $T(w)$, $\{U_{-\beta} : \beta \in R_w\}$ and $\{U^+_{\alpha_j} : j\in J_1\}$.
 
 \end{corollary}
\begin{proof}
  Let $K$ be the kernel of the homomorphism $g_w$. 
 Then, we have the following exact
 sequence of algebraic groups:
 $$1 \longrightarrow K \longrightarrow Aut^0(Z(w_0, \underline j)) \longrightarrow Aut^0(Z(w, \underline i))\longrightarrow 1.$$
By using the differentials, we have following exact sequence of Lie algebras: 
$$ 0\longrightarrow Lie(K)\longrightarrow  H^0(Z(w_0, \underline j), T_{(w_0, \underline j)})
\longrightarrow H^0(Z(w, \underline i), T_{(w, \underline i)}) \longrightarrow 0.$$

By \cite[p.85, Theorem 12.5]{Hum2},
the Lie algebra of $K$ is $Ker (f_w)$. By Corollary \ref{kernel}, we have 
$$Ker (f_w)= (\bigcap_{k\in A_w} Ker (\alpha_k)) \oplus (\bigoplus_{\beta\in R_w}\mathfrak g_{-\beta})\oplus (\bigoplus_{j\in J_1} \mathfrak g_{\alpha_j}).$$

Let $H$ be the closed subgroup of 
$Aut^0(Z(w_0, \underline j))$ generated by $T(w)$, $\{U_{-\beta} : \beta \in R_w\}$ and $\{U^+_{\alpha_j} : j\in J_1\}$.
Note that $H$ is connected (see \cite[p.56, Corollary 7.5]{Hum2}) and
$Lie(H)\subset Ker (f_w)$. Since $dim(Lie(H))=dim(Ker (f_w))$, we have $$Lie(H)=Ker (f_w).$$

Hence, we conclude that $K^0=H$. This completes the proof of the corollary.
\end{proof}

In the following corollary, for the simplicity of notation we denote the  
homogeneous vector bundle $\mathcal L(w, \mathbb C_{\alpha_0})$ on $X(w)$ corresponding to the character $\alpha_0$  of $B$ 
  by $\mathcal L_{\alpha_0}$.

Consider the left action of $T$ on $G/B$. Let $w\in W$. Note that the Schubert variety 
$X(w^{-1})$ is $T$-stable. We use the notion of semi-stable points introduced by Mumford 
\cite{Mum}.  
Let $\alpha_{0}$ be the highest root of $G$ with respect to $T$ and $B^{+}$. 
We denote by $X(w^{-1})_{T}^{ss}(\mathcal{L}_{\alpha_0})$ the set of all semi-stable 
points of $X(w^{-1})$ with respect to the $T$-linearized line bundle 
$\mathcal{L}_{\alpha_0}$ corresponding to the character $\alpha_0$ of $B$ (see \cite{Mum}).

 The following result is a formulation of the Theorem \ref{cor3} using semi-stable points.
 
  \begin{corollary}\label{semi}
\begin{enumerate}
 \item $Aut^0(Z(w, \underline i))$ contains a closed subgroup isomorphic to 
 $ P_{J(w, \underline i)}$ if and only if $ X(w^{-1})^{ss}_T(\mathcal L_{\alpha_0})\neq 
 \emptyset$.
 \item If $G$ is simply laced,  
 $ Aut^0(Z(w, \underline i))\simeq  P_{J(w, \underline i)}$ if and only if $ X(w^{-1})^{ss}_
 T(\mathcal L_{\alpha_0})\neq \emptyset$.

\end{enumerate}
     \end{corollary}
\begin{proof}
 By \cite[Lemma 2.1]{KP}, we have $ X(w^{-1})^{ss}_T(\mathcal L_{\alpha_0})\neq \emptyset$ if and
 only if $w^{-1}(\alpha_0)<0$. 
Proof of the corollary follows from Theorem \ref{cor3} (2) and Theorem \ref{cor3} (3).
 \end{proof}

Remark: By Theorem \ref{cor3}, the automorphism group of the BSDH-variety 
$Z(w, \underline i)$ depends on the choice of the reduced expression $\underline i$ of $w$.

{\bf Example:} Let $G=PSL(4, \mathbb C)$. Consider the following different reduced expressions 
for $w_0$:
\begin{enumerate}
 \item $(w_0, \underline i_1)=s_1s_2s_1s_3s_2s_1$, $J(w_0, \underline i_1)= \{\alpha_1\}$.
 \item$(w_0, \underline i_2)=s_2s_1s_2s_3s_2s_1$, $J(w_0, \underline i_2)= \{\alpha_2\}$.
 \item$(w_0, \underline i_3)=s_3s_2s_3s_1s_2s_3$, $J(w_0, \underline i_3)= \{\alpha_3\}$.
\item$(w_0, \underline i_4)=s_1s_3s_2s_3s_1s_2$, $J(w_0, \underline i_4)= \{\alpha_1, \alpha_3\}$.
 \end{enumerate}

 By Theorem \ref{cor3},  we see that  
$Aut^0(Z(w_0, \underline i_1))$, $Aut^0(Z(w_0, \underline i_2))$, $Aut^0(Z(w_0, \underline i_3))$ 
and \\ $Aut^0(Z(w_0, \underline i_4))$ are 
isomorphic to $P_{\{\alpha_{1}\}}, P_{\{\alpha_{2}\}}, P_{\{\alpha_{3}\}},
 P_{\{\alpha_{1}, \alpha_{3}\}}$ respectively.

Therefore, we observe that $Aut^0(Z(w_0, \underline i_1))$ and $Aut^0(Z(w_0, \underline i_4))$ are 
 not isomorphic and hence we conclude that the BSDH-varieties 
 $Z(w_0, \underline i_1)$ and $Z(w_0, \underline i_4)$ are not isomorphic.
 Also, we observe that $Z(w_0, \underline i_1)$ and $Z(w_0, \underline i_2)$ are not 
 isomorphic as $P_{\{\alpha_1\}}$ and 
 $P_{\{\alpha_2\}}$ are not isomorphic.
 
 {\bf Remark:}
 Even if the automorphism groups of the BSDH-varieties are 
 isomorphic, it is not clear that the BSDH-varieties are isomorphic.

{\bf Acknowledgements.} We would like to thank Professor M. Brion for pointing 
out this problem and for the useful discussions that we had with him. We are very grateful to 
the referees for their valuable comments and suggestions which  helps to improve the paper.
The first  and second named authors would like to thank Infosys Foundation for partial support.

\end{document}